\documentclass[11pt]{amsart}
\usepackage[margin=1in]{geometry}

\usepackage{amssymb}
\usepackage{amsthm}
\usepackage{amsmath}
\usepackage{mathrsfs}
\usepackage{amsbsy}
\usepackage[all]{xy}
\usepackage{bm}
\usepackage{hyperref}
\usepackage{tikz}
\usepackage{array}
\usepackage{float}
\usepackage{enumerate}
\usepackage{xcolor}
\usepackage{hhline}
\allowdisplaybreaks
\usepackage[noadjust]{cite}

\usepackage{caption}
\usepackage{tabu}
\usepackage{diagbox}

\usepackage[noabbrev,capitalise]{cleveref}

\newenvironment{enumerate*}%
  {\begin{enumerate}[(I)]%
    \setlength{\itemsep}{10pt}%
    \setlength{\parskip}{0pt}}%
  {\end{enumerate}}

\newtheorem{theorem}{Theorem}[section]
\newtheorem{proposition}[theorem]{Proposition}
\newtheorem{corollary}[theorem]{Corollary}
\newtheorem{conjecture}[theorem]{Conjecture}

\newtheorem{lemma}[theorem]{Lemma}

\theoremstyle{definition}

\newtheorem{remark}[theorem]{Remark}
\newtheorem{example}[theorem]{Example}

\newcommand{\dfn}[1]{\textcolor{blue}{\emph{#1}}}

\newcommand{\SortNoop}[1]{}

\usepackage{doi}

\DeclareMathOperator{\SRT}{RTab}

\DeclareMathOperator{\cont}{cont}
\DeclareMathOperator{\TYT}{YT}
\DeclareMathOperator{\SW}{\mathsf{Ab}}
\DeclareMathOperator{\Compress}{\mathsf{Compress}}

\begin{document}

\title{Tilings of Benzels via the Abacus Bijection}
\subjclass[2010]{}

\author[Colin Defant]{Colin Defant}
\address[]{Department of Mathematics, Massachusetts Institute of Technology, Cambridge, MA 02139, USA}
\email{colindefant@gmail.com}

\author[Rupert Li]{Rupert Li}
\address[]{Massachusetts Institute of Technology, Cambridge, MA 02139, USA}
\email{rupertli@mit.edu}

\author[James Propp]{James Propp}
\address[]{Department of Mathematical Sciences, UMass Lowell, Lowell, MA 01854, USA}
\email{jamespropp@gmail.com}

\author[Benjamin Young]{Benjamin Young}
\address[]{Department of Mathematics, University of Oregon, Eugene OR 97403 USA}
\email{bjy@uoregon.edu}

\maketitle

\begin{abstract}
    Propp recently introduced regions in the hexagonal grid called \emph{benzels} and stated several enumerative conjectures about the tilings of benzels using two types of prototiles called \emph{stones} and \emph{bones}. We resolve two of his conjectures and prove some additional results that he left tacit. In order to solve these problems, we first transfer benzels into the square grid. One of our primary tools, which we combine with several new ideas, is a bijection (rediscovered by Stanton and White and often attributed to them although it is considerably older) between $k$-ribbon tableaux of certain skew shapes and certain $k$-tuples of Young tableaux. 
\end{abstract}

\section{Introduction}\label{SecIntro}

In \cite{ConwayLagarias}, Conway and Lagarias studied some tiling problems in which the regions that are to be tiled are, like the tiles themselves, composed of regular hexagons in a hexagonal grid; such regions are sometimes called \dfn{polyhexes}. Using combinatorial group theory in a clever and novel fashion, they were able to prove necessary and sufficient conditions for tileability. Later, Thurston \cite{Thurston} provided alternative perspectives and further results, and (more relevantly to our work here) Lagarias and Romano \cite{LagariasRomano} determined the exact number of tilings for a particular one-parameter family of polyhex tiling problems. Meanwhile, physicists had worked in an essentially equivalent (more precisely, dual) setting, studying \dfn{trimer covers} \cite{VN} of the regular 6-valent planar graph, but the physicists' interests were in asymptotic enumeration, not exact formulas for finite subgraphs of the 6-valent grid; furthermore, their proofs relied on the controversial Bethe ansatz, which tends to give correct answers but has not been rigorously established in all the contexts in which it has been applied.

In \cite{Propp2022trimer}, inspired by the large and still-growing literature on exact enumeration of tilings (see \cite{ProppSurvey} for an overview), Propp proposed using the tiles studied by Conway and Lagarias to tile different sorts of regions in the hexagonal grid not considered in earlier literature, and he made numerous conjectures regarding the exact number of tilings of those regions.  We prove some of Propp's conjectures here.

It is convenient to view the hexagonal grid as a tiling of the complex plane with regular hexagons of side length $1$. We can uniquely specify the exact positions of the unit hexagons, which we call the \dfn{cells} of the grid, by declaring that one of the hexagons has vertices $1 \pm 1,1\pm\omega,1\pm\omega^2$, where $\omega=e^{2\pi i/3}$.
Suppose $a$ and $b$ are positive integers satisfying $2\leq a\leq 2b$ and $2\leq b\leq 2a$. Following \cite{Propp2022trimer}, we define the \dfn{$(a,b)$-benzel} to be the union of the cells lying entirely within the hexagon with vertices $a\omega+b,-a\omega^2-b,a\omega^2+b\omega,-a-b\omega,a+b\omega^2,-a\omega-b\omega^2$.
(This is the polyhex that we will attempt to tile using polyhexes consisting of three hexagonal cells, aka \dfn{trihexes}.)
Note that this hexagon is centered at $0$, is invariant under rotation by $120^\circ$, and has sides whose lengths alternate between $2a-b$ and $2b-a$.
\cref{Fig1} shows the $(7,8)$-benzel. It can be shown (see \cite{Propp2022pentagonal}) that the number of cells in the $(a,b)$-benzel is 
\begin{equation}\label{Eq:BenzelSize}
    \begin{cases} \frac{-a^2+4ab-b^2-a-b}{2} & \text{if } a + b \equiv 0 \text{ or }2 \pmod{3}, \\
    \frac{-a^2+4ab-b^2-a-b+2}{2} & \text{if } a + b \equiv 1 \pmod{3}. \end{cases}
\end{equation}

\begin{figure}[ht]
  \begin{center}{\includegraphics[height=7cm]{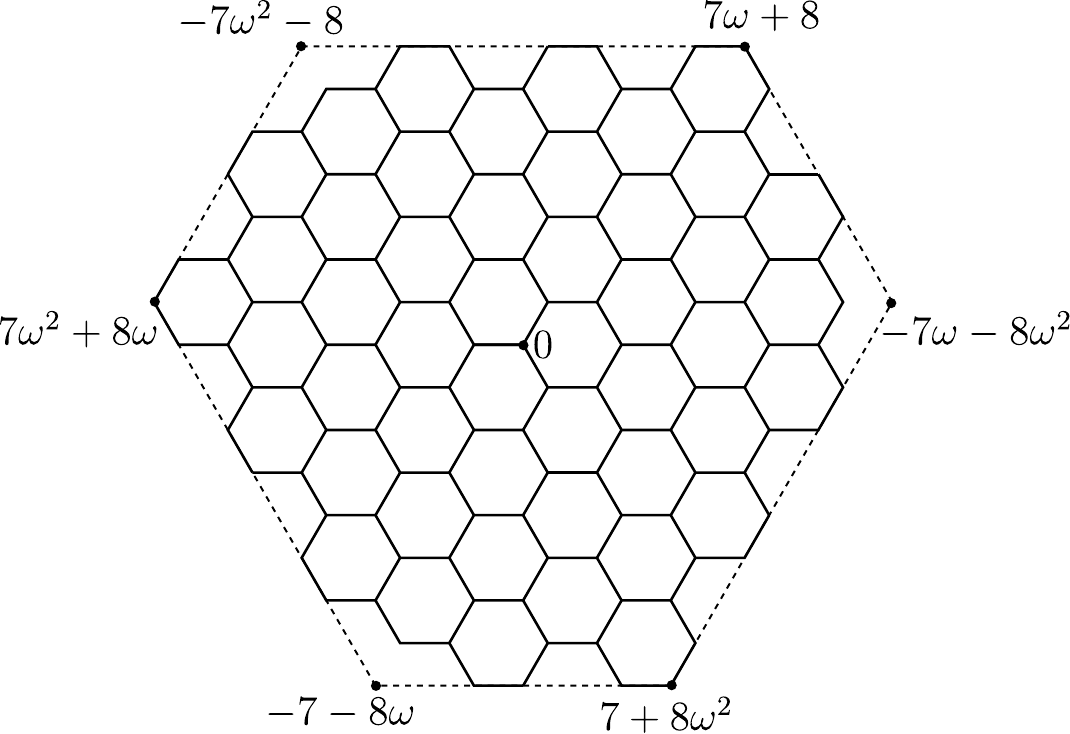}}
  \end{center}
  \caption{The $(7,8)$-benzel.}\label{Fig1}
\end{figure}

\begin{figure}[ht]
  \begin{center}{\includegraphics[height=2.467cm]{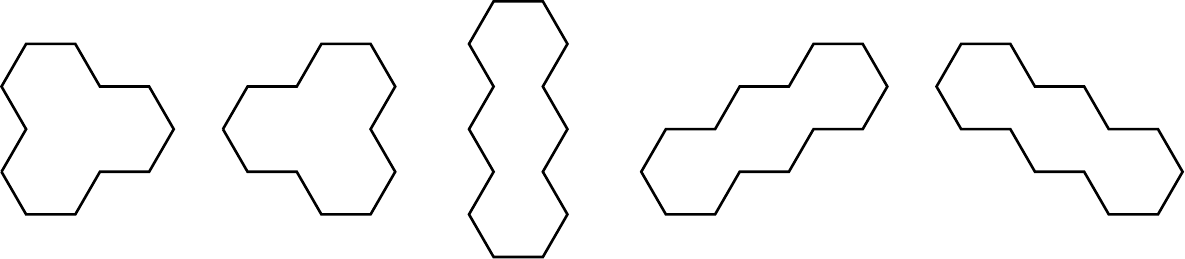}}
  \end{center}
  \caption{The right stone, left stone, vertical bone, rising bone, and falling bone, from left to right, respectively.}\label{Fig2}
\end{figure}

The prototiles in \cite{ConwayLagarias} and \cite{Propp2022trimer} consist of three cells and are of two forms: a \dfn{stone} consists of three hexagonal cells arranged in a triangle while a \dfn{bone} consists of three hexagonal cells arranged in a line.
(Conway and Lagarias called them $T_2$'s and $L_3$'s and Thurston \cite{Thurston} called them $T_2$'s and \emph{tribones}.) We will consider the various translationally-inequivalent rotations of these shapes to be distinct prototiles. Thus, there are really five different prototiles, which are shown in \cref{Fig2}. We call them the \dfn{right stone}, \dfn{left stone}, \dfn{vertical bone}, \dfn{rising bone}, and \dfn{falling bone}. 
Tilings of regions in the hexagonal grid using these prototiles will be referred to as \dfn{stones-and-bones tilings}.

Conway and Lagarias showed that for any simply-connected region in the hexagonal grid that admits a stones-and-bones tiling, the number of right stones minus the number of left stones in such a tiling is invariant (i.e., depends only on the region and not the specific tiling).
This is the \dfn{Conway--Lagarias invariant} of the region.
Propp \cite{Propp2022pentagonal} showed that the Conway--Lagarias invariant of the $(a,b)$-benzel is
\begin{align*}
    \begin{cases}
    \frac{3a^2-6ab+3b^2-a-b}{6} & \text{if } a+b\equiv0\pmod{3}, \\
    \frac{-a^2+4ab-b^2-a-b+2}{6} & \text{if } a+b\equiv1\pmod{3}, \\
    \frac{3a^2-6ab+3b^2+a+b-2}{6} & \text{if } a+b\equiv2\pmod{3}.
    \end{cases}
\end{align*}
Note that when $a+b\equiv1\pmod{3}$, the number of cells in the $(a,b)$-benzel is three times its Conway--Lagarias invariant, which implies (because each tile uses $3$ cells) that a stones-and-bones tiling must consist entirely of right stones. We will prove that such a tiling exists and is unique. 

\begin{theorem}\label{theorem:1mod3}
Let $a$ and $b$ be integers such that $2\leq a\leq 2b$, $2\leq b\leq 2a$, and $a+b\equiv1\pmod{3}$. The $(a,b)$-benzel has a unique stones-and-bones tiling, and this tiling consists entirely of right stones.
\end{theorem}

Frequently, tiling problems will restrict the set of possible prototiles.
There are $2^5-1=31$ nonempty subsets of our set of 5 prototiles. However, as benzels have threefold rotational symmetry---which preserves the two stones---the types of bones allowed does not affect the answer to enumerative questions; all that matters is the number of different bones that are allowed.
Hence, there are really only $2\cdot2\cdot4-1=15$ inequivalent tiling problems to consider.
Moreover, the $(a,b)$-benzel and $(b,a)$-benzel are reflections of each other across the real axis; as this reflection preserves the two stones as well as the number of types of allowed bones, each tiling problem is the same for the $(a,b)$-benzel as the $(b,a)$-benzel. We will often make use of this symmetry in order to assume without loss of generality that $a \leq b$.

Much of our work centers around the enumeration of tilings of benzels in which we allow just two types of bones and one type of stone. Since we only care about the number of allowed bones, we may assume that we are using rising bones and falling bones. Propp conjectured (see \cite[Problems~2 and~3]{Propp2022trimer}) that if $(a,b)$ is of the form $(3n,3n)$ or $(3n+1,3n+2)$, then the number of tilings of the $(a,b)$-benzel using left stones, rising bones, and falling bones is $(2n)!!=2^n n!$. One of our main results resolves this conjecture; in fact, we will determine the number of such tilings of the $(a,b)$-benzel for all choices of $a$ and $b$. 

\begin{theorem}[cf.~\protect{\cite[Problems 2 and 3]{Propp2022trimer}}]\label{thm:mountainless}
Let $a$ and $b$ be integers with $2\leq a\leq b\leq 2a$. If there is an integer $n$ such that $(a,b)=(3n,3n)$ or $(a,b)=(3n+1,3n+2)$, then there are exactly $(2n)!!$ tilings of the $(a,b)$-benzel by left stones, rising bones, and falling bones. Otherwise, there are no such tilings of the $(a,b)$-benzel.  
\end{theorem}

On the other hand, we have the following theorem regarding tilings with right stones, rising bones, and falling bones.

\begin{theorem}\label{thm:valleyless}
Let $a$ and $b$ be integers with $2\leq a\leq b\leq 2a$. If $b=2a$ or $a+b\equiv 1\pmod 3$, then there is a unique tiling of the $(a,b)$-benzel by right stones, rising bones, and falling bones. If $b<2a$ and $a+b\equiv 0\pmod 3$, then there are no such tilings.  
\end{theorem}

\cref{thm:valleyless} says nothing about what happens when $a+b\equiv 2\pmod 3$.; Propp \cite{Propp2022trimer} gave the following conjecture in this case.
\begin{conjecture}[\protect{\cite[Problem 5]{Propp2022trimer}}]\label{conj}
Let $a$ and $b$ be integers with $2\leq a\leq b\leq 2a$ and $a+b\equiv 2\pmod 3$. Let $n$ and $k$ be the integers such that $(a,b)=(n+3k,2n+3k-1)$. The $(a,b)$-benzel has
\[ \prod_{i=1}^k \frac{(2i)!(2i+2n-2)!}{(i+n-1)!(i+n+k-1)!} \]
tilings by right stones, rising bones, and falling bones. 
\end{conjecture}

Our approach to proving \cref{thm:mountainless,thm:valleyless} is to first transfer benzels from the hexagonal grid into the square grid. This allows us to use tools from the theory of ribbon tilings of Young diagrams. In particular, if $\lambda$ is a partition with $k$-core $\kappa$ and $k$-quotient $(\lambda^{(0)},\ldots,\lambda^{(k-1)})$, then there is a bijection (dubbed the ``abacus bijection'' by Gordon James \cite{JamesKerber} and rediscovered by Dennis Stanton and Dennis White (see \cite{StantonWhite} and \cite{FominStanton}) that maps $k$-ribbon tableaux of shape $\lambda/\kappa$ to $k$-tuple Young tableaux of shape $(\lambda^{(0)},\ldots,\lambda^{(k-1)})$. For other authors' descriptions of the bijection, see \cite{JamesKerber} or \cite{Pak}. We felt the need to provide our own discussion and proof; we hope that our treatment of this bijection will be useful to researchers.
More relevant to our purpose is a related bijection that we call $\Compress$, which is apparently new. This bijection is between $k$-ribbon tilings of a Young diagram of a particular type and $(k-1)$-tilings of a related Young diagram; it is essentially a version of the abacus bijection that ``forgets'' the ordering of the tiles in a ribbon tableau. 
Our proofs of \cref{thm:mountainless,thm:valleyless} employ these bijections as tools and combine them with several new ideas.
Although we have made no attempt toward proving \cref{conj}, we expect that it could be fruitful to transfer the problem into the square grid and employ techniques similar to those that we use.

We mention in passing that, like us, Conway and Lagarias transferred their tiling problem to the square grid. However, there are two key differences between their use of the tactic and ours.
The first difference is that, as Thurston later demonstrated, the use of the square grid is not an essential feature of the translation of tiling problems into combinatorial group theory problems; in contrast, our adaptation of the theory of cores and quotients requires working in the square grid (where Ferrers graphs and Young diagrams live).
The second difference is that in \cite{ConwayLagarias}, all three bones are transferred to the square grid, at the cost of making one of them disconnected; in contrast, such disconnected tiles are forbidden in the theory of ribbon tilings.

The outline for the rest of the paper is as follows. \cref{Sec:Square} is brief; its purpose is to explain how to translate certain tiling problems in the hexagonal grid into tiling problems in the square grid.
In \cref{Sec:Abacus}, we provide a thorough discussion of $k$-cores, $k$-quotients, and the abacus bijection. \Cref{SecCompress} defines and develops the $\Compress$ bijection. \cref{Sec1mod3} is devoted to proving \cref{theorem:1mod3}.
We use the properties of $\Compress$ discussed in \cref{SecCompress} to prove \cref{thm:mountainless} in \cref{SecMountainless} and to prove \cref{thm:valleyless} in \cref{SecValleyless}.

\section{Transferring Tilings to the Square Grid}\label{Sec:Square}

In this article, we draw the square grid in the complex plane so that the unit square cells are oriented like diamonds. More precisely, each square grid cell has vertices of the form $(a\pm 1)+bi$ and $a+(b\pm 1)i$, where $a$ and $b$ are integers such that $a+b$ is odd. 
We call each such square cell a \dfn{box}.
A great deal of theory has already been developed for tilings of regions in the square grid; in order to make use of it, we will need a way of converting tilings in the hexagonal grid into tilings in the square grid.
Notice that there are unit-width vertical strips of the complex plane that are traversed only by horizontal edges of the hexagonal grid, not by edges of the other two orientations.
Removing these vertical strips and compressing the plane appropriately yields a bijective mapping from the hexagons in the hexagonal grid to rhombuses in the resulting rhombic grid.
Rescaling the axes suitably transforms these rhombuses into squares of side length $\sqrt 2$. Let us then rotate the resulting grid by $90^\circ$ counterclockwise.

\begin{figure}[ht]
  \begin{center}{\includegraphics[height=3.281cm]{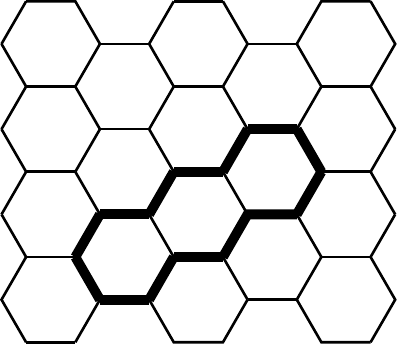}\raisebox{1.5cm}{$\xrightarrow{\text{delete strips}}$}\includegraphics[height=3.281cm]{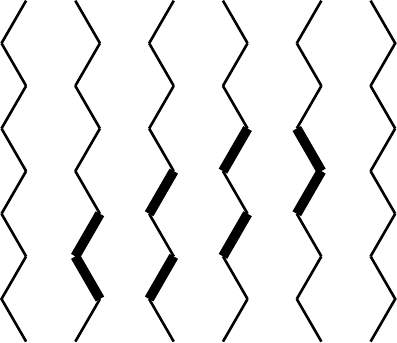}\raisebox{1.5cm}{$\xrightarrow{\text{compress}}$}\includegraphics[height=3.281cm]{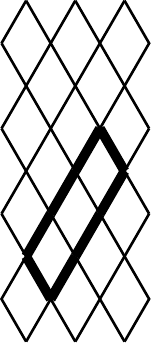}\raisebox{1.5cm}{$\xrightarrow{\text{scale and rotate}}$}\raisebox{.88cm}{\includegraphics[height=1.452cm]{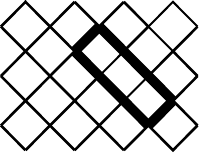}}}
  \end{center}
  \caption{The steps that transform the hexagonal grid into the square grid. The bold rising bone in the hexagonal grid becomes a negative bone in the square grid.}\label{Fig3}
\end{figure}

The prototiles of Figure \ref{Fig2}, transferred to this square grid, are shown in \cref{Fig4}.
Notice that the vertical bone becomes a union of three disconnected boxes.
When at most two types of bones are allowed, we can assume without loss of generality that the vertical bone is forbidden, which allows the resulting square grid tiling problem to exclusively use prototiles of connected boxes.
In fact, when the vertical bone is forbidden, we can specifically leverage the theory of \emph{ribbon tilings}.
A \dfn{ribbon} is a connected union of boxes in the square grid in which no two boxes have the same $x$-coordinate; a \dfn{$k$-ribbon} is a ribbon that contains exactly $k$ unit squares. 
Our four prototiles (excluding the vertical bone) in the square grid are precisely the four $3$-ribbons.
With these conventions, we will rename the right stone, left stone, rising bone, and falling bone the \dfn{mountain stone}, \dfn{valley stone}, \dfn{negative bone}, and \dfn{positive bone}.

\begin{figure}[ht]
  \begin{center}{\includegraphics[height=1.324cm]{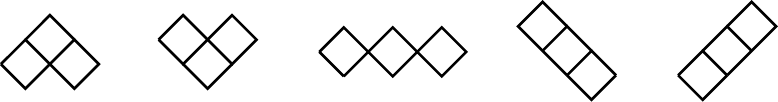}}
  \end{center}
  \caption{Transforming the prototiles in \cref{Fig2} into the square grid yields these five prototiles. From left to right, the first, second, fourth, and fifth are the mountain stone, valley stone, negative bone, and positive bone, respectively. }\label{Fig4}
\end{figure}

\section{The Abacus Bijection}\label{Sec:Abacus}

The version of the bijection that we use is due to Gordon James (see \cite{JamesKerber}) but different forms of it seem to have been discovered independently by various people working in the field of modular representation theory around 1950; this community includes H.\ Farahat, J.\ S.\ Frame, D.\ E.\ Littlewood, T.\ Nakayama, M.\ Osima, G.\ de B.\ Robinson, R.\ A.\ Staal, and R.\ M.\ Thrall. The interested reader may find more details in the book \cite{Robinson} and the references it contains.

We identify integer partitions with their Young diagrams. We will draw the Young diagram of a partition $\lambda$ in Russian notation so that it lives within the square grid as we have chosen to draw it.
Let us position the Young diagram so that its bottom-most point is $0$ and so that the boxes representing the first part of the partition have their lower-left edges lying along the ray $\mathbb R_{\geq 0}\,e^{3\pi i/4}$, as in Figure~\ref{Fig5}.

A \dfn{coordinatized abacus word} is a function $w$ from the set $\mathbb Z+1/2=\{n+1/2:n\in\mathbb Z\}$ to the alphabet $\{\circ,\bullet\}$ with the property that $w(-n-1/2)=\bullet$ and $w(n+1/2)=\circ$ for all sufficiently large $n$.
Such words also arise in the study of the exclusion process; in that literature, $\bullet$ is called a \emph{particle} and $\circ$ is called a \emph{vacancy}. We can represent such a word as a bi-infinite sequence of the symbols $\circ$ and $\bullet$, where we place a period between the symbols $w(-1/2)$ and $w(1/2)$ to indicated the position $0$. For example, $\cdots\bullet\bullet\bullet\circ\bullet\!\!\hspace{0.015in}.\hspace{0.015in}\!\!\bullet\circ\circ\circ\cdots$ represents the coordinatized abacus word $w$ given by $w(-3/2)=w(n+1/2)=\circ$ for all integers $n\geq 1$ and $w(1/2)=w(-1/2)=w(m-1/2)$ for all integers $m\leq -2$. 
An \dfn{abacus word} is an orbit of a coordinatized abacus word under the natural $\mathbb Z$-action given by the shift map. As with coordinatized abacus words, we can represent abacus words as bi-infinite sequences using the symbols $\circ$ and $\bullet$; however, we no longer include the period in this representation.  

Given a partition $\lambda$, we obtain an abacus word $\overline w_\lambda$ by traveling along its northern border from left to right and recording whether we go up or down at each step by writing the symbol $\circ$ to record an up step and the symbol $\bullet$ to record a down step. From this abacus word $\overline w_\lambda$, we obtain a coordinatized abacus word $w_\lambda$ by insisting that $w_\lambda(\ell)$ records whether the step whose midpoint has real part $\ell$ is up or down.
For example, if $\lambda=(5,5,3,3,2)$ is the partition shown in \cref{Fig5}, then $w_\lambda=\cdots\bullet\bullet\bullet\circ\circ\bullet\bullet\circ\hspace{0.02pt}.\!{}\hspace{-0.9pt}{}\circ\bullet\circ\bullet\bullet\circ\circ\circ\cdots$. 

Each abacus word $\overline w$ (which is, by definition, an orbit under the shift map) has a unique representative of the form $w_\lambda$ for some partition $\lambda$.
Indeed, in the partition $\lambda$, there must be the same number of up steps to the left of position $0$ as down steps to the right of position $0$, so
\[ |\{\ell < 0 : w_\lambda(\ell)=\circ\}| = |\{\ell \geq 0 : w_\lambda(\ell)=\bullet\}|; \]
this value is also the number of boxes of $\lambda$ that are directly above the origin. We call $w_\lambda$ the \dfn{canonical coordinatization} of $w$.

\begin{figure}[ht]
  \begin{center}{\includegraphics[height=3cm]{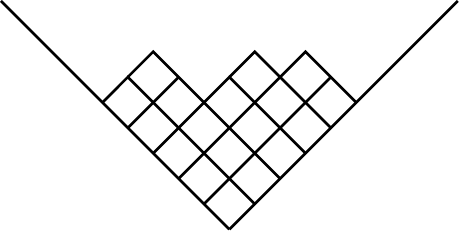}}
  \end{center}
  \caption{The Young diagram of the partition $\lambda=(5,5,3,3,2)$. We have $w_\lambda=\cdots\bullet\bullet\bullet\circ\circ\bullet\bullet\circ.\!\circ\bullet\circ\bullet\bullet\circ\circ\circ\cdots$.} \label{Fig5}
\end{figure}

Let us fix a positive integer $k$ and split the word $w_\lambda$ into $k$ words $w^{(0)}_\lambda,\ldots,w^{(k-1)}_\lambda$ so that each $w^{(j)}_\lambda$ is obtained by reading every $k$-th symbol in $w_\lambda$. In other words, $w_\lambda^{(j)}$ is the word obtained by reading $w_\lambda(k\ell+j+1/2)$ for all $\ell\in\mathbb Z$. As above, we can coordinatize each of the words $w_\lambda^{(j)}$ so that it corresponds to an integer partition $\lambda^{(j)}$. For each $0 \leq j\leq k-1$, there is a unique integer $c_j$ such that if we set $w^{(j)}_\lambda(\ell-c_j+1/2)=w_\lambda(k\ell+j+1/2)$ for all $\ell$, then $w_\lambda^{(j)}=w_{\lambda^{(j)}}$,
the abacus word associated with the partition $\lambda^{(j)}$.
The integers $c_0,\ldots,c_{k-1}$ are called the \dfn{$k$-charges} of $\lambda$; they satisfy $c_0+\cdots+c_{k-1}=0$. The tuple $(\lambda^{(0)},\ldots,\lambda^{(k-1)})$ is called the \dfn{$k$-quotient} of $\lambda$. For an example of this construction, let $k=3$, and let $\lambda=(5,5,3,3,2)$ as shown in \cref{Fig5}. We have $w_{\lambda}^{(0)}=\cdots\bullet\bullet\bullet\circ\bullet\circ\circ\circ\cdots$, $w_{\lambda}^{(1)}=\cdots\bullet\bullet\bullet\circ\bullet\bullet\bullet\circ\circ\circ\cdots$, and $w_{\lambda}^{(2)}=\cdots\bullet\bullet\bullet\circ\circ\circ\cdots$. Then $\lambda^{(0)}=(1)$, $\lambda^{(1)}=(3)$, and $\lambda^{(2)}=\emptyset$ are shown in \cref{Fig6}. The $3$-charges are $c_0=c_1=1$ and $c_2=-2$.

\begin{figure}[ht]
  \begin{center}{\includegraphics[height=1.677cm]{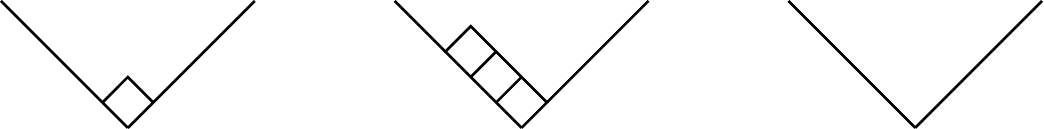}}
  \end{center}
  \caption{The $3$-quotient of the partition $\lambda=(5,5,3,3,2)$ is the triple $(\lambda^{(0)},\lambda^{(1)},\lambda^{(2)})=((1),(3),\emptyset)$.
  The $3$-charges are $c_0=c_1=1$ and $c_2=-2$.}
  \label{Fig6}
\end{figure}

Suppose there is a $k$-ribbon tile $T$ at the top of the Young diagram of $\lambda$ such that removing $T$ results in the Young diagram of a smaller partition. Let us remove it. Iterating this process, we can continue removing $k$-ribbon tiles until it is no longer possible to do so. The resulting partition $\kappa$ is called the \dfn{$k$-core} of $\lambda$; it is known that $\kappa$ does not depend on the order in which the $k$-ribbons were removed. 
We will consider the skew shape $\lambda/\kappa$. It will be convenient to think of $\kappa$ as a tile inside of $\lambda$.

A \dfn{$k$-ribbon tableau}  of shape $\lambda/\kappa$ is a tuple $(T_1,\ldots,T_m)$ of $k$-ribbon tiles such that the set $\{\kappa,T_1,\ldots,T_m\}$ forms a tiling of $\lambda$ and such that for every $1\leq i\leq m$, the set $\{\kappa, T_1,\ldots,T_i\}$ forms a tiling of some Young diagram. One can imagine adding the tiles $\kappa,T_1,\ldots,T_m$ one by one to build $\lambda$ from the bottom up. Let $\SRT_k(\lambda/\kappa)$ denote the set of $k$-ribbon tableaux of shape $\lambda/\kappa$. 

A \dfn{strict Young tableau} of shape $\lambda$ is a filling of the square boxes of $\lambda$ (which is again drawn in Russian notation) with distinct positive integers so that numbers strictly increase as we move up (northeast or northwest). We write $\cont(\mathcal T)$ for the \dfn{content} of a strict Young tableau $\mathcal T$, which is just the set of integers appearing in $\mathcal T$. 

Suppose $(\alpha_0,\ldots,\alpha_{k-1})$ is a $k$-tuple of partitions. We define a \dfn{$k$-tuple Young tableau} of shape $(\alpha_0,\ldots,\alpha_{k-1})$ to be a $k$-tuple $(\mathcal T_0,\ldots,\mathcal T_{k-1})$ of strict Young tableaux such that each $\mathcal T_j$ has shape $\alpha_j$ and $\cont(\mathcal T_0)\cup\cdots\cup\cont(\mathcal T_{k-1})=[|\alpha_0|+\cdots+|\alpha_{k-1}|]$.
(Here we use $[n]$ to denote $\{1,2,\dots,n\}$.)
Note that these conditions force the sets $\cont(\mathcal T_0),\ldots,\cont(\mathcal T_{k-1})$ to be pairwise disjoint. Let $\TYT_k(\alpha_0,\ldots,\alpha_{k-1})$ be the set of $k$-tuple Young tableaux of shape $(\alpha_0,\ldots,\alpha_{k-1})$. 

As above, let $\lambda$ be a partition with $k$-quotient $(\lambda^{(0)},\ldots,\lambda^{(k-1)})$ and $k$-core $\kappa$. The \dfn{abacus bijection} is a bijection $\SW\colon\SRT_k(\lambda/\kappa)\to\TYT_k(\lambda^{(0)},\ldots,\lambda^{(k-1)})$. To describe it, let us start with a $k$-ribbon tableau $(T_1,\ldots,T_m)\in\SRT_k(\lambda/\kappa)$. The intuitive idea is to remove the tiles $T_m,\ldots,T_1$ in this order and insert the number $r$ into one of the boxes of one of $\lambda^{(0)},\ldots,\lambda^{(k-1)}$ at the point in time when we remove $T_r$. In order to be rigorous, it is convenient to define the bijection recursively.
Thus, let us assume that we have already defined the abacus bijection on $\SRT_k(\mu/\kappa)$ whenever $|\mu/\kappa|<|\lambda/\kappa|$.

Let $\ell\in\mathbb Z$ and $j\in\{0,\ldots,k-1\}$ be such that $k\ell+j$ is the real part of the leftmost point in the $k$-ribbon tile $T_m$; the real part of the rightmost point in $T_m$ must be $k(\ell+1)+j+1$. There are $k+1$ steps in the northern boundary of the Young diagram of $\lambda$ that are part of the $k$-ribbon tile $T_m$; they are encoded by the symbols $w_\lambda(k\ell+j+1/2),\ldots,w_\lambda(k(\ell+1)+j+1/2)$. It is straightforward to see that $w_\lambda(k\ell+j+1/2)=\circ$ and $w_\lambda(k(\ell+1)+j+1/2)=\bullet$. Let $\mu$ be the partition whose Young diagram is obtained from that of $\lambda$ by removing $T_m$.
Then $w_\mu$ agrees with $w_\lambda$ except that $w_\mu(k\ell+j+1/2)=\bullet$ and $w_\mu(k(\ell+1)+j+1/2)=\circ$.
The $k$-charges of $\mu$ are the same as the $k$-charges $c_0,\ldots,c_{k-1}$ of $\lambda$. Indeed, if we replace $\lambda$ by $\mu$, the quantities $|\{\ell<0: w_\lambda^{(j)}(\ell-c_j+1/2)=\circ\}|$ and $|\{\ell\geq0: w_\lambda^{(j)}(\ell-c_j+1/2)=\bullet\}|$ will not change if $\ell\neq -1$ and will both decrease by $1$ if $\ell=-1$; more generally, adding or removing $k$-ribbons from the Young diagram of a partition $\nu$ to yield a new Young diagram of a partition $\nu'$ does not change the $k$-charges.
This implies $\lambda$, $\mu$, and $\kappa$ have the same $k$-charges.
It follows that $\lambda^{(r)}=\mu^{(r)}$ for all $r\neq j$ and that the Young diagram of $\mu^{(j)}$ is obtained from that of $\lambda^{(j)}$ by removing a single box $B$.
Note that the real parts of the leftmost and rightmost points in $B$ are $\ell-c_j$ and $\ell-c_j+2$, respectively.
Since $(T_1,\ldots,T_{m-1})\in\SRT_k(\mu/\kappa)$ and $|\mu/\kappa|<|\lambda/\kappa|$, we can apply the abacus bijection inductively to obtain a $k$-tuple Young tableau $\SW(T_1,\ldots,T_{m-1})\in\TYT_k(\mu^{(0)},\ldots,\mu^{(k-1)})$.
The $k$-tuple Young tableau $\SW(T_1,\ldots,T_m)\in\TYT_k(\lambda^{(0)},\ldots,\lambda^{(k-1)})$ is now obtained from $\SW(T_1,\ldots,T_{m-1})$ by adding the box $B$ and filling it with the number $m$.

\begin{example}
Let $k=3$ and $\lambda=(5,5,3,3,2)$, as in \cref{Fig5}.
We saw above that $\lambda^{(0)}=(1)$, $\lambda^{(1)}=(3)$, and $\lambda^{(2)}=\emptyset$. 
The $3$-core of $\lambda$ is $\kappa=(4,2)$. One standard $3$-ribbon tableau $(T_1,T_2,T_3,T_4)$ of shape $\lambda/\kappa$ is shown at the top of \cref{Fig7}; the corresponding $3$-tuple Young tableau $\SW(T_1,T_2,T_3,T_4)$ appears at the bottom of the same figure. To illustrate the recursive description of the abacus bijection, let $\mu$ be the partition obtained by removing $T_4$ from $\lambda$. Then $w_\lambda=\cdots\bullet\bullet\bullet\circ\circ\bullet\bullet\circ\circ\bullet\circ\bullet\bullet\circ\circ\circ\cdots$ and $w_\mu=\cdots\bullet\bullet\bullet\bullet\circ\bullet\circ\circ\circ\bullet\circ\bullet\bullet\circ\circ\circ\cdots$; the latter is obtained from the former by changing one $\circ$ to a $\bullet$ and changing one $\bullet$ to a $\circ$.
Preserving notation from above, we have $\ell=-2$ and $j=1$. 
\end{example}

\begin{figure}[ht]
  \begin{center}{\includegraphics[height=3cm]{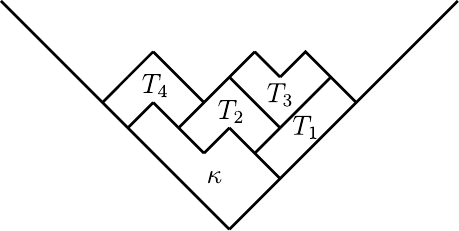}\\\vspace{.5cm}
  \includegraphics[height=1.677cm]{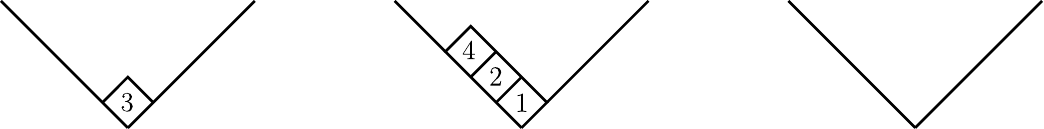}}
  \end{center}
  \caption{An example of the abacus bijection for $\lambda=(5,5,3,3,2)$. At top is a $3$-ribbon tableau of shape $\lambda$; at bottom is the corresponding $3$-tuple Young tableau of shape $((1),(3),\emptyset)$.}
  \label{Fig7}
\end{figure}

To prove that the map $\SW\colon\SRT_k(\lambda/\kappa)\to\TYT_k(\lambda^{(0)},\ldots,\lambda^{(k-1)})$ is bijective, we construct its inverse by reversing the above procedure. Suppose we start with a $k$-tuple Young tableau $(\mathcal T^{(0)},\ldots,\mathcal T^{(k-1)})$ of shape $(\lambda^{(0)},\ldots,\lambda^{(k-1)})$. Let $m=|\lambda^{(0)}|+\cdots+|\lambda^{(k-1)}|$, and let $j\in\{0,\ldots,k-1\}$ be such that $m\in\cont(\mathcal T^{(j)})$. Let $B$ be the box of $\lambda^{(j)}$ containing $m$ in $\mathcal T^{(j)}$.
Let $c_0,\ldots,c_{k-1}$ be the $k$-charges of $\lambda$, and define $\ell$ so that $\ell-c_j$ is the real part of the leftmost point in $B$.
We can reconstruct the partition $\mu^{(j)}$ by removing $B$ from $\lambda^{(j)}$.
We can then reconstruct $\mu$, which is determined by its $k$-quotient $(\mu^{(0)},\ldots,\mu^{(k-1)})$ and $k$-charges $c_0,\ldots,c_{k-1}$ (it has the same $k$-charges as $\lambda$), by letting $\mu^{(r)}=\lambda^{(r)}$ for all $r\neq j$.
Let $\widetilde{\mathcal T}^{(j)}$ be the strict Young tableau of shape $\mu^{(j)}$ obtained by deleting $B$ from $\mathcal T^{(j)}$, and let $\widetilde{\mathcal T}^{(r)}=\mathcal T^{(r)}$ for all $r\neq j$. Then $(\widetilde{\mathcal T}^{(0)},\ldots,\widetilde{\mathcal T}^{(k-1)})$ is a $k$-tuple Young tableau of shape $(\mu^{(0)},\ldots,\mu^{(k-1)})$. As before, the words $w_\lambda$ and $w_\mu$ agree except that $(w_\lambda(k\ell+j+1/2),w_\mu(k\ell+j+1/2))=(\circ,\bullet)$ and $(w_\lambda(k(
\ell+1)+j+1/2),w_\mu(k(\ell+1)+j+1/2))=(\bullet,\circ)$. This implies that $\lambda$ is obtained from $\mu$ by adding a $k$-ribbon tile $T_m$. Hence, $\mu$ and $\lambda$ have the same $k$-core $\kappa$.
We may assume inductively that $\SW\colon\SRT_k(\mu/\kappa)\to\TYT_k(\mu^{(0)},\ldots,\mu^{(k-1)})$ is a bijection, so we can re-obtain the $k$-ribbon tableau $(T_1,\ldots,T_{m-1})\in\SRT_k(\mu/\kappa)$ as $\SW^{-1}(\widetilde{\mathcal T}^{(0)},\ldots,\widetilde{\mathcal T}^{(k-1)})$.
Then $\SW^{-1}(\mathcal T^{(0)},\ldots,\mathcal T^{(k-1)})$ is the $k$-ribbon tableau $(T_1,\ldots,T_m)$.

Finally, we address the base case $|\lambda/\kappa|=0$, i.e., $\lambda=\kappa$.
This implies $\SRT_k(\lambda/\kappa)=\{\emptyset\}$.
If any of $\kappa^{(0)},\ldots,\kappa^{(k-1)}$ are nonempty, then applying our inverse map to remove a box implies that a $k$-ribbon can be removed from $\kappa$ to yield a smaller partition $\kappa'$, contradicting the definition of the $k$-core.
Thus $\kappa^{(0)}=\cdots=\kappa^{(k-1)}=\emptyset$, and $\TYT_k(\lambda^{(0)},\ldots,\lambda^{(k-1)})=\{(\emptyset,\ldots,\emptyset)\}$.
We obtain a trivial bijection $\SW\colon\SRT_k(\lambda/\kappa)\to\TYT_k(\lambda^{(0)},\ldots,\lambda^{(k-1)})$.

This proves the following theorem. 

\begin{theorem}[\cite{JamesKerber}]
Let $\lambda$ be an integer partition with $k$-quotient $(\lambda^{(0)},\ldots,\lambda^{(k-1)})$ and $k$-core $\kappa$.
The map $\SW\colon\SRT_k(\lambda/\kappa)\to\TYT_k(\lambda^{(0)},\ldots,\lambda^{(2)})$ defined above is a bijection. 
\end{theorem}

A $k$-ribbon tableau of shape $\lambda/\mu$ is a $k$-ribbon tiling of $\lambda/\mu$ together with a certain ordering of the tiles. If we are given a $k$-ribbon tiling, then there is at least one way to order the tiles to obtain a $k$-ribbon tableau. Indeed, we can define a partial order $\prec$ on the set of tiles by declaring that $T\prec T'$ whenever $T$ has a box that lies below one of the boxes in $T'$. Then the orderings of the tiles that yield $k$-ribbon tableaux are precisely the linear extensions of this partially ordered set. Our interest in this article is in unordered tilings, so our main motivation for discussing the abacus bijection comes from the following corollary. 

\begin{corollary}\label{Cor:TilingsExist}
Let $\lambda$ be an integer partition with $k$-quotient $(\lambda^{(0)},\ldots,\lambda^{(k-1)})$ and $k$-core $\kappa$. There exists a $k$-ribbon tiling of the Young diagram of $\lambda$ if and only if $\kappa$ is empty, and this occurs if and only if $|\lambda^{(0)}|+\cdots+|\lambda^{(k-1)}|=\frac{1}{k}|\lambda|$. 
\end{corollary}
\begin{proof}
As discussed above, any given skew shape has a $k$-ribbon tableau if and only if it can be tiled by $k$-ribbons. The definition of the $k$-core $\kappa$ implies that it is the smallest partition contained in $\lambda$ such that there is a $k$-ribbon tableau of shape $\lambda/\kappa$. This shows that $\lambda$ can be tiled by $k$-ribbons if and only if $\kappa$ is empty. It follows from the above discussion of the abacus bijection that $|\lambda^{(0)}|+\cdots+|\lambda^{(k-1)}|=\frac{1}{k}(|\lambda|-|\kappa|)$, so $\kappa$ is empty if and only if $|\lambda^{(0)}|+\cdots+|\lambda^{(k-1)}|=\frac{1}{k}|\lambda|$. 
\end{proof}
\begin{remark}\label{remark:coreCharge}
The above discussion of the abacus bijection noted that the $k$-quotient of any $k$-core $\kappa$ is empty, so $\kappa$ is completely determined by the $k$-charges of $\lambda$, which are the same as the $k$-charges of $\kappa$.
In other words, given the $k$-charges of $\lambda$, we can construct its $k$-core $\kappa$ to be the unique partition with an empty $k$-quotient and the same charges as $\lambda$.
In particular, $\kappa$ is empty if and only if the $k$-charges are all zero.
\end{remark}

\section{The Compress Bijection}\label{SecCompress}
Suppose $\lambda$ is a partition with empty $k$-core whose $k$-quotient $(\lambda^{(0)},\dots,\lambda^{(k-1)})$ satisfies $\lambda^{(j)}=\emptyset$.
Let $\rho$ be the partition with empty $(k-1)$-core whose $(k-1)$-quotient is \[(\lambda^{(0)},\dots,\lambda^{(j-1)},\lambda^{(j+1)},\dots,\lambda^{(k-1)}).\]
We will define a bijection $\Compress$ from the set of $k$-ribbon tilings of $\lambda$ to the set of $(k-1)$-ribbon tilings of $\rho$.
This bijection is essentially the unordered tiling equivalent of using the abacus bijection to map a $k$-ribbon tableau to a $k$-tuple Young tableau, removing the empty $\lambda^{(j)}$ to obtain a $(k-1)$-tuple Young tableau, and then lifting via the inverse abacus bijection on $(k-1)$-ribbons back to a $(k-1)$-ribbon tableau.
However, without an ordering on our ribbons, a one-step description of our bijection exists, and it is simpler than this two-step process involving the abacus bijection.

Similar to how we removed vertical strips of the hexagonal grid that were traversed by horizontal edges and then compressed the plane to obtain the square grid in \cref{Sec:Square}, our bijection is obtained by removing the vertical strips of the square grid that correspond to $\lambda^{(k-1)}=\emptyset$ and then compressing the plane to remove the empty space.

Identifying $\lambda$ with its Young diagram, let us write $P_r=\lambda\cap\{z\in\mathbb C:r\leq \operatorname{Re}(z)< r+1\}$, where $\operatorname{Re}(z)$ denotes the real part of $z$.
The assumptions that $\lambda^{(j)}=\emptyset$ and that the $k$-charges of $\lambda$ are all zero (since its $k$-core is empty) guarantee that $P_r$ is a parallelogram whenever $r\equiv j\pmod k$; let us color these parallelograms green.
Imagine deleting these green parallelograms, dividing what remains into vertical bands, sliding the remaining pieces of $\lambda$ lying left of the imaginary axis to the southeast, and sliding the remaining pieces of $\lambda$ lying right of the imaginary axis to the southwest.
(Specifically, tile-pieces in the $i$-th band to the left of the imaginary axis slide $i$ steps southeast, while tile-pieces in the $i$-th band to the right of the imaginary axis slide $i$ steps southeast.)
See \cref{Fig8} for an example when $k=3$ and $j=1$. It is straightforward to check that the resulting shape will be the Young diagram of $\rho$.
The bijection $\Compress$ simply transfers the tiles in a $k$-ribbon tiling of $\lambda$ through this process so that they become $(k-1)$-ribbons that tile $\rho$. It is not difficult to check that if we start with a $(k-1)$-ribbon tiling $\mathcal D$ of $\rho$, then there is a unique way to lift it to a $k$-ribbon tiling $\mathcal T$ of $\lambda$ such that $\Compress(\mathcal T)=\mathcal D$; every $(k-1)$-ribbon has a unique square that will be extended into two squares, where the direction of extension depends on which side of the imaginary axis the square is on, yielding a $k$-ribbon.
Hence, we only need to check that $\Compress$ is well-defined; more precisely, we need to show that the remove-green compression process actually sends all of the tiles in a $k$-ribbon tiling of $\lambda$ to $(k-1)$-ribbons in $\rho$ (instead of disconnected shapes).

\begin{figure}[ht]
  \begin{center}\includegraphics[height=3cm]{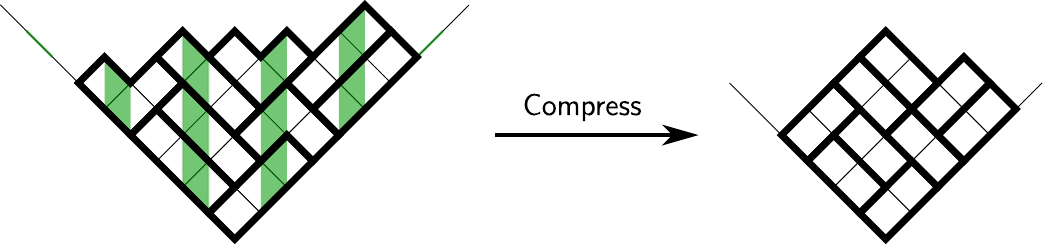}\end{center}
  \caption{Applying the bijection $\Compress$ to the $3$-ribbon tiling of the partition $\lambda$ on the left yields the $2$-ribbon tiling of the partition $\rho$ on the right. We have also drawn the green parallelograms in $\lambda$.}
  \label{Fig8}
\end{figure}

Let us fix a $k$-ribbon tiling $\mathcal T$ of $\lambda$.
Suppose we order the tiles in $\mathcal T$ to form a $k$-ribbon tableau.
When we apply the abacus bijection to this tableau, each tile will correspond to one of the boxes in the resulting $k$-tuple Young tableau.
More precisely, a tile $T$ will correspond to one of the boxes in the Young tableau of shape $\lambda^{(r)}$, where $r\in\{0,1,\dots,k-1\}$ is congruent modulo $k$ to the real part of the leftmost point in $T$.
Since $\lambda^{(j)}=\emptyset$, we cannot have $r=j$.
Equivalently, $T$ cannot intersect two of the green parallelograms.
It must intersect at least one of the green parallelograms because the real parts of the leftmost and rightmost points of $T$ are $k$ apart and because the horizontal distance between two consecutive green parallelograms is only $k-1$.
This proves the following lemma. 

\begin{lemma}\label{lem:no_green_to_green}
Let $\lambda$ be a partition with empty $k$-core and with $k$-quotient $(\lambda^{(0)},\dots,\lambda^{(k-1)})$ satisfying $\lambda^{(j)}=\emptyset$.
If $\mathcal T$ is a $k$-ribbon tiling of $\lambda$, then each tile in $\mathcal T$ intersects a unique green parallelogram. 
\end{lemma}

Now choose a fixed green parallelogram $P=P_r$ in $\lambda$.
The side lengths of $P$ are $\sqrt 2$ and $2\alpha$ for some nonnegative integer $\alpha$.
We can break $P$ into $\alpha$ smaller parallelograms whose side lengths are $\sqrt 2$ and $2$; say these smaller parallelograms are $P^{1},\ldots,P^\alpha$, listed from bottom to top.
Let $L^s$ and $R^s$ be the left and right vertical sides of $P^s$, respectively.
Each tile in $\mathcal T$ that intersects $P$ must contain exactly one of $L^1,\ldots,L^\alpha$ and exactly one of $R^1,\ldots,R^\alpha$.
It follows that for each $1\leq s\leq \alpha$, there is a tile $T^s$ that contains both $L^s$ and $R^s$.
When we apply the compression process removing the green parallelograms, the piece of $T^s$ to the left of $P$ will connect with the piece of $T^s$ to the right of $P$ to form a $(k-1)$-ribbon, as desired.
See \cref{Fig11} for an illustration when $k=3$. 

\begin{figure}[ht]
  \begin{center}
  \includegraphics[height=5cm]{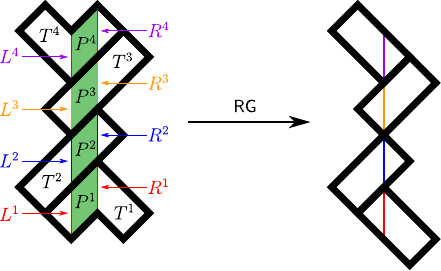}
\end{center}
  \caption{The remove-green compression process turns $k$-ribbons that intersect the green parallelogram into $(k-1)$-ribbons.}
  \label{Fig11}
\end{figure}

\begin{remark}\label{rem:correct_slope}
Suppose $T$ is a tile in a $k$-ribbon tiling $\mathcal T$ of $\lambda$. By \cref{lem:no_green_to_green}, $T$ intersects a unique green parallelogram $P$. In $P$ appears to the left (respectively, right) of the imaginary axis, then it follows from the preceding discussion that the line segments in the the boundary of $T$ that intersect $P$ must have negative (respectively, positive) slope. \Cref{lem:no_green_to_green} and the present remark encompass the main properties of the $\mathsf{Compress}$ bijection that we will use in \Cref{SecMountainless,SecValleyless}. 
\end{remark}

\begin{remark}\label{remark:remove_different_colors}
The preceding discussion implies that the Young diagrams with $k$-quotients \[(\mu^{(0)},\mu^{(1)},\dots,\mu^{(k-2)},\emptyset), (\mu^{(0)},\dots,\mu^{(k-3)},\emptyset,\mu^{(k-2)}), \ldots, (\emptyset,\mu^{(0)},\dots,\mu^{(k-2)})\] with vanishing $k$-charges all have the same number of $k$-ribbon tilings, as each such set of tilings is in bijection with the set of $(k-1)$-ribbon tilings of the Young diagram with $(k-1)$-quotient $(\mu^{(0)},\dots,\mu^{(k-2)})$ and vanishing $(k-1)$-charges.
\end{remark}

\begin{remark}\label{remark:remove_multiple_colors}
The $\Compress$ bijection can be generalized to remove multiple empty $k$-quotients simultaneously.
Removing $k$-quotients $\lambda^{(j_1)}=\cdots=\lambda^{(j_n)}=\emptyset$, we simply remove all $P_r$ with $r\equiv j_\ell \pmod{k}$ for some $\ell$ and then slide the remaining pieces of $\lambda$ together accordingly.
This is equivalent to composing $n$ $\Compress$ bijections together, each time removing one of the empty parts of the quotient.
\end{remark}

The relationship between the abacus bijection and the compress bijection is summarized in the following commutative diagram: 
\[
\includegraphics[width=\linewidth]{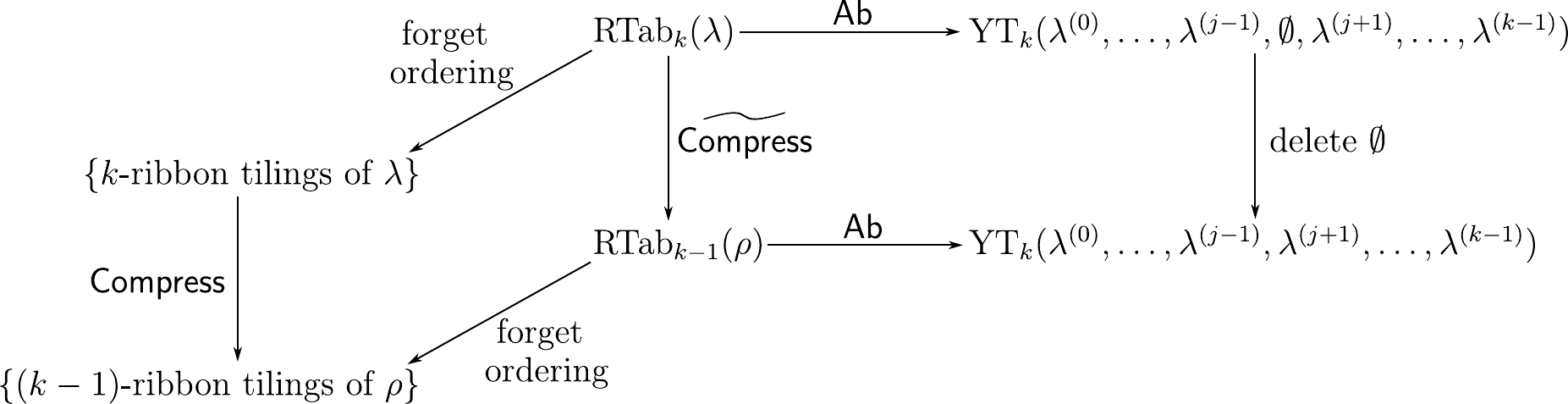}
\]
The map from $\SRT_k(\lambda)$ (respectively, $\SRT_{k-1}(\rho)$) to the set of $k$-ribbon tilings of $\lambda$ (respectively, $(k-1)$-ribbon tilings of $\rho$) simply forgets the ordering of the tiles.
The map $\widetilde{\mathsf{Compress}}$ acts in the same way as the map $\Compress$, except it preserves the ordering of the tiles as it compresses $k$-ribbon tiles into $(k-1)$-ribbon tiles.

\section{Tilings of Benzels Using Only Right Stones}\label{Sec1mod3}
When $a+b\equiv1\pmod{3}$, we can use the Conway--Lagarias invariant to see that in any stones-and-bones tiling of the $(a,b)$-benzel (allowing all five types of stones and bones), the benzel must be tiled entirely by right stones. In this section, we prove \cref{theorem:1mod3}, which states that such a tiling exists and is unique. 

\begin{proof}[Proof of \cref{theorem:1mod3}]
The $(2,2)$-benzel consists of three hexagonal cells in the shape of a right stone, so the theorem is obvious when $a=b=2$. Thus, we may assume $a+b>4$ and proceed by induction on $a+b$. By the discussion above, we simply need to show that there is only one way to tile the $(a,b)$-benzel with right stones. 
The $(a,b)$-benzel and the $(b,a)$-benzel are related by a reflection across the real axis, and this reflection preserves right stones. Therefore, it suffices to consider the case when $a\leq b$.  

We illustrate the argument in \cref{Fig17}. The bottom side of the bounding hexagon of the $(a,b)$-benzel touches $\frac{2a-b+1}{3}$ of the hexagonal cells of the benzel. Any tiling of the $(a,b)$-benzel with right stones must cover these $\frac{2a-b+1}{3}$ cells with distinct right stones, which are aligned in a row. Since benzels and stones have $120^\circ$-rotational symmetry, the same argument shows that the upper left side and the upper right side of the benzel must also each be tiled with a line of $\frac{2a-b+1}{3}$ right stones. Removing these three lines of right stones results in the $(a,b-3)$-benzel, which we know by induction has a unique tiling by right stones. 
\end{proof}

\begin{figure}[ht]
  \begin{center}
 \includegraphics[height=5.274cm]{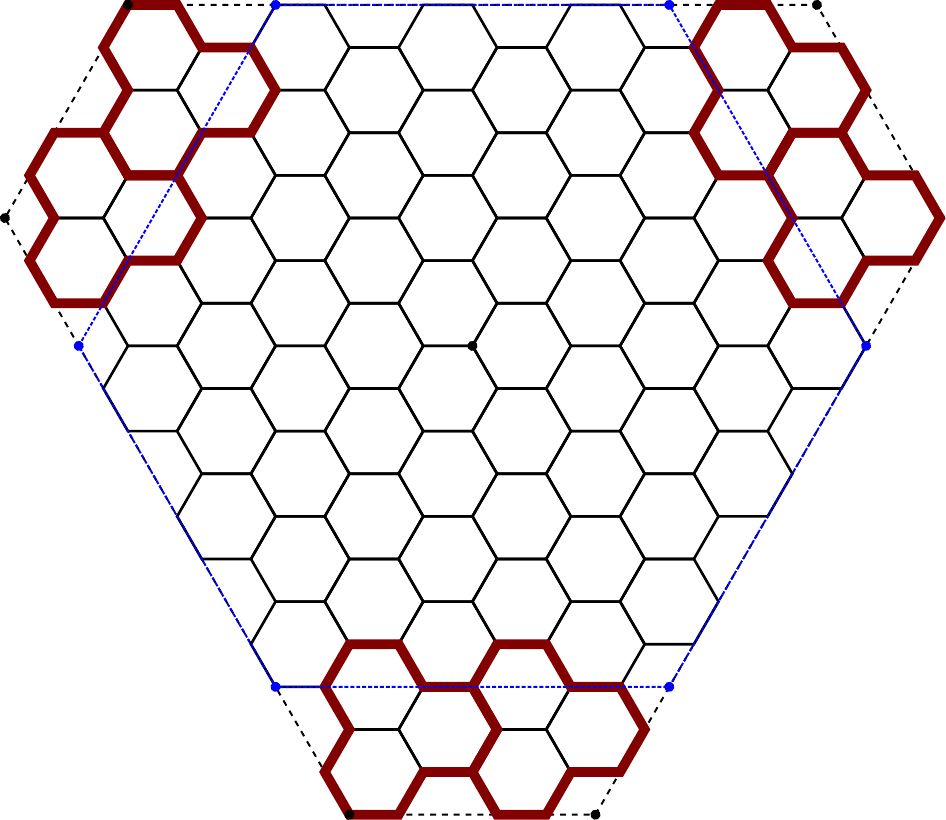} 
  \end{center}
  \caption{The $(8,11)$-benzel. We have indicated the three lines of right stones appearing on the bottom and in the upper left and upper right corners of the benzel; each line contains $\frac{2\cdot 8-11+1}{3}=2$ right stones.  Removing these lines of right stones results in the $(8,8)$-benzel, as indicated by the blue hexagon. }\label{Fig17}
\end{figure}

This immediately implies that for any tiling problem where only certain specified prototiles are allowed, there is a unique tiling of the $(a,b)$-benzel for $a+b\equiv1\pmod{3}$ if and only if the right stone is an allowed prototile, and no tilings otherwise.

\begin{remark}\label{remark:boundaryCase}
Recall that the parameters $a,b$ of an $(a,b)$-benzel must satisfy $2 \leq a \leq 2b$ and $2 \leq b \leq 2a$.
For the three residues of $a+b$ modulo 3, the boundary cases are the $(n,2n-2)$-benzel, the $(n,2n-1)$-benzel, and the $(n,2n)$-benzel for $n\geq2$.
For each $n$, these three benzels coincide; as \cref{theorem:1mod3} applies to the $(n,2n-2)$-benzel, it also implies that the $(n,2n-1)$-benzel and $(n,2n)$-benzel each have a unique stones-and-bones tiling, which consists entirely of right stones.
\end{remark}

\begin{remark}\label{remark:phases}
It should be noted that all the stones
in the unique tiling constructed above are in phase with each other.
That is, if we 3-color the cells in the hexagonal grid so that no two adjoining hexagons have the same color (which can be done in an essentially unique way), then all the stones in the tiling constructed above exhibit the same color-pattern.
Putting it differently, the tiling of the benzel extends to a periodic tiling of the plane by right stones whose fundamental domain can be chosen to be a single tile.
\end{remark}

\section{Two Bones and the Left Stone}\label{SecMountainless}

The goal of this section is to enumerate tilings of the $(a,b)$-benzel using two types of bones along with the left stone. Without loss of generality, we may assume the vertical bone is forbidden. 
Thus, our goal is to prove \cref{thm:mountainless}, which provides the enumeration of tilings using rising bones, falling bones, and left stones. It follows immediately from \cref{theorem:1mod3} that there are no such tilings when $a+b\equiv1\pmod{3}$ (since the only stones-and-bones tiling of the $(a,b)$-benzel uses only right stones). If $a+b\equiv2\pmod{3}$, the Conway--Lagarias invariant of the $(a,b)$-benzel is
\[ \frac{1}{6}(3a^2-6ab+3b^2+a+b-2) = \frac{1}{6}\left(3(a-b)^2+a+b-2\right) \geq \frac{a+b-2}{6} \geq \frac{1}{3} > 0. \]
Therefore, in this case, a stones-and-bones tiling of the $(a,b)$-benzel requires a strictly positive number of right stones, so again no such tiling exists. This completes the proof of \cref{thm:mountainless} when $a+b\not\equiv 0\pmod 3$, so we may assume throughout the rest of this section that $a+b\equiv 0\pmod 3$. 

After transferring the problem to the square grid as in \cref{Sec:Square}, we are left to consider tilings using valley stones, negative bones, and positive bones, i.e., the three prototiles other than the mountain stone.
Henceforth, we will refer to such tilings with these three allowed prototiles as \dfn{mountainless tilings}.

We will use the $\Compress$ bijection from \cref{SecCompress} with $k=3$, i.e., between certain 3-ribbon tilings and 2-ribbon tilings (also called \emph{domino tilings}) of smaller shapes. Note that there are two 2-ribbons, the 2-ribbon analogs of the positive and negative bones.
We will set $j=1$, so our bijection will be from 3-ribbon tilings of an integer partition $\lambda$ with 3-quotient $(\lambda^{(0)},\emptyset,\lambda^{(2)})$ and vanishing $3$-charges. As we will see later, the integer partition that ends up being relevant for our purposes is the partition $\lambda_n$ with 3-quotient $(\square_n,\emptyset,\square_n)$ and vanishing $3$-charges, where $\square_n=(n,n,\dots,n)$ is the integer partition consisting of $n$ parts of size $n$. The Young diagram of $\lambda_n$ consists of $n$ nested V-shapes $V_1,\ldots,V_n$, listed from top to bottom. For each $1\leq i\leq n-1$, we will assign the color red to the boundary between $V_i$ and $V_{i+1}$; we call this boundary a \dfn{red border}. See \cref{Fig9} for an example with $n=3$. Let us label the boxes in $\lambda_n$ with the letters A, B, C so that the label of a box is determined by reading the real part of its leftmost point modulo $3$, with A, B, C corresponding to $0$, $1$, $2$, respectively. 

\begin{figure}[ht]
  \begin{center}\includegraphics[height=3.888cm]{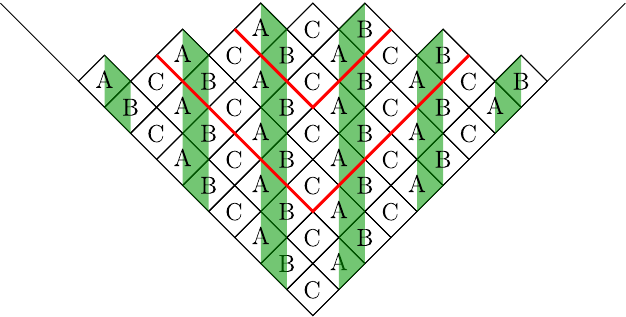} \end{center}
  \caption{The Young diagram of the partition $\lambda_3$ breaks up into three V-shaped regions $V_1,V_2,V_3$, which are separated by red borders. We have drawn the green parallelograms and labeled each box with A, B, or C.}
  \label{Fig9}
\end{figure}

\begin{lemma}\label{lem:valleys_C}
If $\mathcal T$ is a $3$-ribbon tiling of $\lambda_n$, then no bones in $\mathcal T$ cross any red borders. If $v$ is a valley stone in $\mathcal T$ that crosses a red border, then there is exactly one box in $v$ lying above that red border, and this one box has label $\mathrm{C}$. If $w$ is a mountain stone in $\mathcal T$ that crosses a red border, then there is exactly one box in $w$ lying below that red border, and this one box has label $\mathrm{C}$.
\end{lemma}

\begin{proof}
By \cref{lem:no_green_to_green}, every tile must intersect a unique green parallelogram. Furthermore, \cref{rem:correct_slope} tells us that each tile must intersect its green parallelogram with the correct slope. This immediately implies that no bone may cross a red border. 

Now suppose there is a valley stone $v$ that crosses one of the red borders. Because this stone cannot intersect two green parallelograms, its center box must have label A or B. If the label is A, then the center box must appear to the right of the imaginary axis since, otherwise, the valley stone would cross the green parallelogram with the wrong slope (contradicting \cref{rem:correct_slope}). Similarly, if the center box has label B, then it must appear to the left of the imaginary axis. In either case, there is exactly one box in $v$ lying above the red border that $v$ crosses, and this one box must have label C. 

A completely analogous argument proves the last statement of the lemma concerning the mountain stone $w$, except for one new possibility that must be considered: a mountain stone whose top cell is the bottom cell of one of the V-shapes. However, this possibility can be eliminated by noting that such a mountain stone intersects two green parallelograms, one on each side, violating \cref{lem:no_green_to_green}.
\end{proof}

\begin{proposition}\label{prop:lambda_n_(2n)!!}
The number of mountainless tilings of $\lambda_n$ is $(2n)!!$.
\end{proposition}
\begin{proof}
Consider a mountainless tiling of $\lambda_n$.
We claim that every tile must be contained in a single V-shaped region $V_m$; in other words, no tile can cross a red border. To see this, suppose by way of contradiction that our mountainless tiling contains at least one tile that crosses a red border. Let $t$ be such that the highest red border that is crossed is the one between $V_t$ and $V_{t+1}$, and let $X=V_1\cup\cdots\cup V_t$. According to \cref{lem:valleys_C}, the only stones that cross the red border between $V_t$ and $V_{t+1}$ are valley stones. Furthermore, each such valley stone uses exactly one box from $X$, and each such box is labeled C. However, each tile lying entirely within $X$ has exactly one box labeled A, one box labeled B, and one box labeled C. This is a contradiction because one can readily check that there are equal numbers of boxes labeled A, B, and C lying in $X$.

We conclude the proof by showing that $V_m$ has $2m$ mountainless tilings.
We may tile $V_m$ by placing a single valley stone as far left as possible (e.g., the leftmost blue silhouette in \cref{Fig10}), and then tiling the remaining region with bones.
This implies the Conway--Lagarias invariant of the region $V_m$ is $-1$, so any mountainless tiling of $V_m$ must use exactly one valley stone.
It is straightforward to see that there are $2m$ positions for this one valley stone (those centered at a cell labeled B on the left half or at a cell labeled A on the right half) that admit a tiling of the remaining region with bones; moreover, each of these positions for the valley stone gives rise to a unique mountainless tiling.
For example, the 6 positions for the valley stone in $V_3$ are shown in \cref{Fig10}.
Placing a valley stone in any other position would split $V_m$ into two parts such that the number of cells in each part is not divisible by $3$, so it would not yield a tiling. Thus, there are $2m$ mountainless tilings of $V_m$, which implies (since no tile can cross a red border) that $\lambda$ has $(2n)!!$ mountainless tilings. 
\end{proof}

\begin{figure}[ht]
  \begin{center}\includegraphics[height=3.217cm]{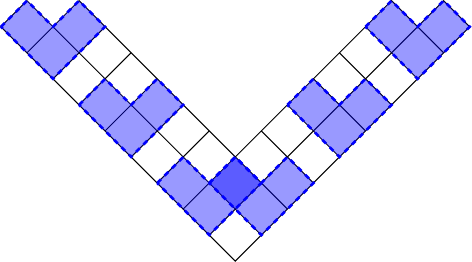} \end{center}
  \caption{There are $6$ positions where a valley stone could appear in a mountainless tiling of $V_3$ (two of these positions overlap). }
  \label{Fig10}
\end{figure}

Recall that we are assuming $2\leq a\leq b\leq 2a$ and $a+b\equiv 0\pmod 3$; say $a+b=3(N+1)$ with $N \geq 1$. Let $B_{a,b}$ be the region obtained by transferring the $(a,b)$-benzel into the square grid as in \cref{Sec:Square}. 
We can embed $B_{a,b}$ inside of the Young diagram of $\lambda_N$.
More precisely, it is not hard to see that there are exactly three boxes of maximal height in $B_{a,b}$, corresponding to the three rightmost cells of the $(a,b)$-benzel, and there are exactly three boxes of maximal height in $\lambda_N$. There is a unique way to embed $B_{a,b}$ into $\lambda_N$ so that the three maximal-height boxes in $B_{a,b}$ coincide with those of $\lambda_N$. \cref{Fig12,Fig13} illustrate this for $(a,b)=(7,8)$ and $(a,b)=(8,10)$, respectively; in each of these figures, the embedded image of $B_{a,b}$ is colored in blue in the diagram on the right. 

\begin{figure}[ht]
  \begin{center}
\includegraphics[height=4.179cm]{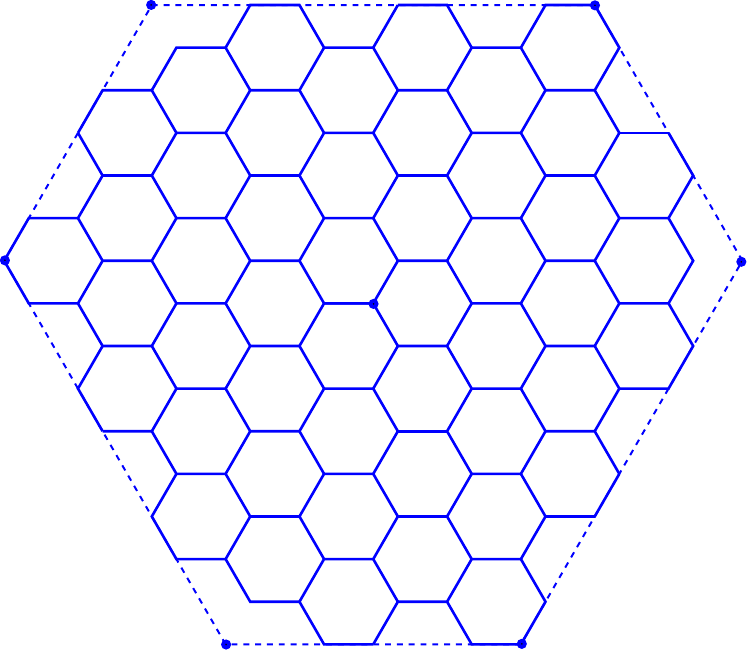}\qquad\qquad\includegraphics[height=4.011cm]{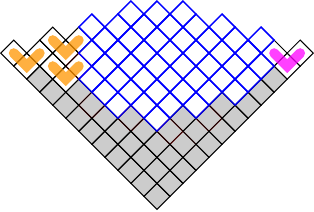} 
 \end{center}
  \caption{The $(7,8)$-benzel (left) and its embedded image $B_{7,8}$ in $\lambda_4$ (right), shown in blue. There are $L=2$ columns of orange valley stones on the left side of $\lambda_4$ and $R=1$ column of pink valley stones on the right side of $\lambda_4$. The region obtained by removing $B_{7,8}$ and the orange and pink valley stones is the Young diagram of $\theta_{7,8}$, which is shaded gray. }\label{Fig12}
\end{figure}

We can divide the region $\lambda_N\setminus B_{a,b}$ into three pieces. The first piece is a region to the left of $B_{a,b}$ that can be tiled by $\binom{L+1}{2}$ valley stones, where $L=\frac{2b-a}{3}-1$; these valley stones (shown in orange in Figures~\ref{Fig12} and~\ref{Fig13}) are arranged in columns of sizes $1,2,\ldots,L$, from left to right. The second piece is a region to the right of $B_{a,b}$ that can be tiled by $\binom{R+1}{2}$ valley stones, where $R=\frac{2a-b}{3}-1$; these valley stones (shown in pink in Figures~\ref{Fig12} and~\ref{Fig13}) are arranged in columns of sizes $R,R-1,\ldots,1$, from left to right. The third piece is the Young diagram of a partition $\theta_{a,b}$ (shaded in Figures~\ref{Fig12} and~\ref{Fig13}). It is most convenient to describe the partition $\theta_{a,b}$ via its abacus word, which is \begin{equation}\label{Eq:theta}
    w_{\theta_{a,b}}=\cdots\bullet\bullet\bullet(\circ\bullet\bullet)^{L}\bullet(\circ\bullet\bullet)^{R}(\circ\circ\bullet)^{L}\circ(\circ\circ\bullet)^{R}\circ\circ\circ\cdots.
\end{equation}

Recall that the \dfn{Durfee square} of a partition $\mu$ is the $s\times s$ square Young diagram inside of $\mu$, where $s$ is the largest integer such that $\mu$ has at least $s$ parts of size at least $s$. We call $s$ the \dfn{size} of the Durfee square. The following lemma, which we record for later use, can be read off immediately from the abacus word in \eqref{Eq:theta} (using the fact that $L+R=N-1$). 

\begin{lemma}\label{lem:Durfee}
Let $a$ and $b$ be integers with $2\leq a\leq b\leq 2a$ and $a+b\equiv 0\pmod 3$. The Durfee square of $\theta_{a,b}$ has size $N-1$, where $N=\frac{a+b}{3}-1$. 
\end{lemma}

Our motivation for embedding $B_{a,b}$ into $\lambda_{N}$ comes from the following lemma, which states that there exists a $3$-ribbon tiling of $\lambda_N\setminus B_{a,b}$. This will allow us to extend $3$-ribbon tilings of $B_{a,b}$ to $3$-ribbon tilings of $\lambda_N$, which can then be analyzed via the abacus bijection and the remove-green bijection $\Compress$. 

\begin{lemma}\label{lem:tiling_the_complement}
Let $a$ and $b$ be integers with $2\leq a\leq b\leq 2a$ and $a+b\equiv 0\pmod 3$. Let $N=\frac{a+b}{3}-1$. As above, let $B_{a,b}$ be the embedded image of the $(a,b)$-benzel inside of $\lambda_N$. There exists a $3$-ribbon tiling of the region $\lambda_N\setminus B_{a,b}$. 
\end{lemma}

\begin{proof}
Let $L=\frac{2b-a}{3}-1$ and $R=\frac{2a-b}{3}-1$, and note that $N=L+R+1$. As above, we can divide $\lambda_N\setminus B_{a,b}$ into three pieces, two of which we know can be tiled using valley stones. Thus, it suffices to show that there exists a $3$-ribbon tiling of the remaining piece $\theta_{a,b}$. By \cref{Cor:TilingsExist}, we just need to check that $|\theta_{a,b}^{(0)}|+|\theta_{a,b}^{(1)}|+|\theta_{a,b}^{(2)}|=\frac{1}{3}|\theta_{a,b}|$, where $(\theta_{a,b}^{(0)},\theta_{a,b}^{(1)},\theta_{a,b}^{(2)})$ is the $3$-quotient of $\theta_{a,b}$. Alternatively, we could appeal to \cref{remark:coreCharge} and show that the $3$-charges of $\theta_{a,b}$ are all zero. From the abacus word $w_{\theta_{a,b}}$ given in \eqref{Eq:theta}, we can compute the abacus words of the partitions in the $3$-quotient of $\theta_{a,b}$: 
\begin{align*}
  w_{\theta_{a,b}}^{(0)}&=\cdots\bullet\bullet\bullet\circ^{L}\bullet^{N}\circ\circ\circ\cdots,  \\
  w_{\theta_{a,b}}^{(1)}&=\cdots\bullet\bullet\bullet\circ^{N}\bullet^{R}\circ\circ\circ\cdots, \\ 
  w_{\theta_{a,b}}^{(2)}&=\cdots\bullet\bullet\bullet\bullet^{L}\circ^{N}\circ\circ\circ\cdots=\cdots\bullet\bullet\bullet\circ\circ\circ\cdots.
\end{align*}
It is now straightforward to check that the $3$-charges of $\theta_{a,b}$ are all $0$, which completes the proof. 

For completeness, let us also outline an argument that uses \cref{Cor:TilingsExist}. We can see directly from the three abacus words listed above that the partitions in the $3$-quotient of $\theta_{a,b}$ are rectangles and that the number of boxes in each of these partitions is given by  $|\theta_{a,b}^{(0)}|=LN$, $|\theta_{a,b}^{(1)}|=RN$, and $|\theta_{a,b}^{(2)}|=0$. Because $\lambda_N$ has vanishing $3$-charges and $3$-quotient $(\square_N,\emptyset,\square_N)$, it follows from \cref{Cor:TilingsExist} that $|\lambda_N|=3(|\square_N|+|\emptyset|+|\square_N|)=6N^2$. We know by \eqref{Eq:BenzelSize} that $|B_{a,b}|=\frac{-a^2+4ab-b^2-a-b}{2}$. There are $\binom{L+1}{2}$ orange valley stones in the left piece of $\lambda_N$, and there are $\binom{R+1}{2}$ pink valley stones in the right piece. Thus, \[|\theta_{a,b}|=6N^2-\frac{-a^2+4ab-b^2-a-b}{2}-3\binom{L+1}{2}-3\binom{R+1}{2};\] one can verify that this equals $3LN+3RN$. Thus, $|\theta_{a,b}^{(0)}|+|\theta_{a,b}^{(1)}|+|\theta_{a,b}^{(2)}|=LN+RN+0=\frac{1}{3}|\theta_{a,b}|$. 
\end{proof}

\begin{figure}[ht]
  \begin{center}
 \includegraphics[height=5cm]{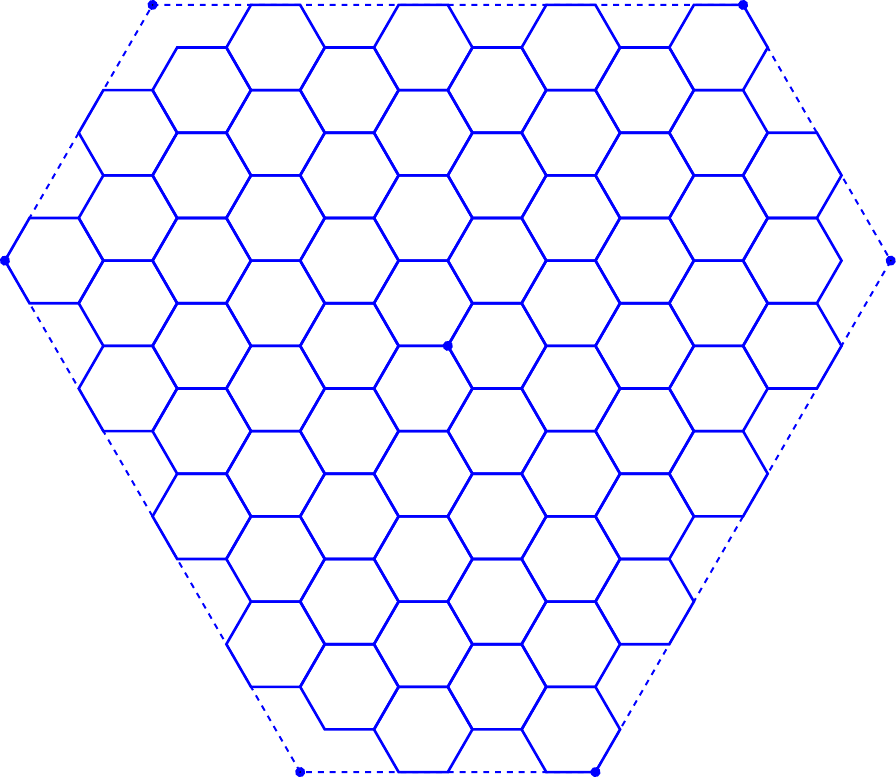}\qquad\qquad\includegraphics[height=5cm]{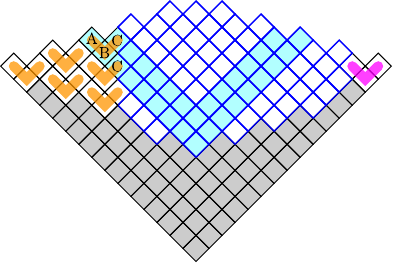} 
  \end{center}
  \caption{The $(8,10)$-benzel (left) and its embedded image $B_{8,10}$ in $\lambda_5$ (right), shown in blue. There are $L=3$ columns of orange valley stones on the left side of $\lambda_5$ and $R=1$ column of pink valley stones on the right side of $\lambda_5$. The region obtained by removing $B_{8,10}$ and the orange and pink valley stones is the Young diagram of $\theta_{8,10}$, which is shaded gray. The V-shaped region $V_{N-L+1}=V_{3}$ is shaded in light blue. There are exactly $4$ boxes in $V_3$ that lie outside of $B_{8,10}$, and they are marked with their labels A, B, C, and C.}\label{Fig13}
\end{figure}

We can now prove another piece of \cref{thm:mountainless}.

\begin{proposition}\label{prop:b-aatleast2}
Let $a$ and $b$ be integers with $2\leq a\leq b\leq 2a$, $a+b\equiv 0\pmod 3$, and $b-a\geq 2$. There are no mountainless tilings of $B_{a,b}$. 
\end{proposition}

\begin{proof}
As above, let us embed $B_{a,b}$ into $\lambda_N$, where $N=\frac{a+b}{3}-1$. We will need to consider the labeling of the boxes in $\lambda_N$ with the letters A, B, and C as before. Let $L=\frac{2b-a}{3}-1$ and $R=\frac{2a-b}{3}-1$, and consider the $L$ columns of orange valley stones and the $R$ columns of pink valley stones as before. 

The smallest positive integer $m$ such that $V_m$ intersects an orange valley stone is $N-L+1$, and the smallest positive integer $m$ such that $V_m$ intersects a pink valley stone is $N-R+1$. The condition $a<b$ implies that $L>R$.  Therefore, $V_{N-L+1}$ intersects some of the orange valley stones, but it does not intersect any of the pink valley stones. More precisely, there are exactly $4$ boxes that belong to $V_{N-L+1}$ and orange valley stones; their labels are A, B, C, and C. See \cref{Fig13} for an example with $a=8$ and $b=10$. We claim that all of the boxes in $V_{N-L+1}$ other than these $4$ are contained in $B_{a,b}$. This is equivalent to the claim that the lowest box in $V_{N-L+1}$ is not in the Durfee square of $\theta_{a,b}$. We know by \cref{lem:Durfee} that the Durfee square of $\theta_{a,b}$ has size $N-1$, so the number of boxes in $\lambda_N$ that are centered on the imaginary axis and lie above the Durfee square of $\theta_{a,b}$ is $2N-(N-1)=N+1$. Hence, our claim is equivalent to the inequality $2(N-L+1)\leq N+1$, which is equivalent to our hypothesis that $b-a\geq 2$. 

Suppose by way of contradiction that $\mathcal T$ is a mountainless tiling of $B_{a,b}$. It follows from the claim established in the previous paragraph that $B_{a,b}\cap (V_1\cup\cdots\cup V_{N-L+1})$ has strictly fewer boxes labeled C than boxes labeled A. Since each $3$-ribbon tile uses exactly one box of each label, this implies that there must be a tile $v$ in $\mathcal T$ that has a box labeled A in $B_{a,b}\cap (V_1\cup\cdots\cup V_{N-L+1})$ and has a box labeled C in $B_{a,b}\setminus (V_1\cup\cdots\cup V_{N-L+1})$.
\cref{lem:tiling_the_complement} tells us that there exists a $3$-ribbon tiling of $\lambda_N\setminus B_{a,b}$, so we can extend $\mathcal T$ to a (not necessarily mountainless) $3$-ribbon tiling $\mathcal T^*$ of $\lambda_N$. Then $v$ is a tile in $\mathcal T^*$ that crosses the red border between $V_{N-L+1}$ and $V_{N-L+2}$, and it is not a mountain stone because it belongs to the mountainless tiling $\mathcal T$. Applying \cref{lem:valleys_C} to the tiling $\mathcal T^*$, we find that $v$ must be a valley stone that has a box labeled C in $V_{N-L+1}$. This is a contradiction because the only box labeled C in $v$ is in $B_{a,b}\setminus (V_1\cup\cdots\cup V_{N-L+1})$. 
\end{proof}

We are now left to consider the cases $(a,b)=(3n,3n)$ and $(a,b)=(3n+1,3n+2)$.
Define $F_n$ to be the region in the square grid (considered modulo translation) obtained by placing $n-1$ columns of negative bones next to each other, where all bones in the same column have leftmost points with the same real part and the $m$-th column counted from the left has $n-m$ bones. For example, \cref{Fig14} shows the regions $F_2,F_3,F_4$, and $F_5$.

\begin{figure}[ht]
  \begin{center}{\includegraphics[height=3.254cm]{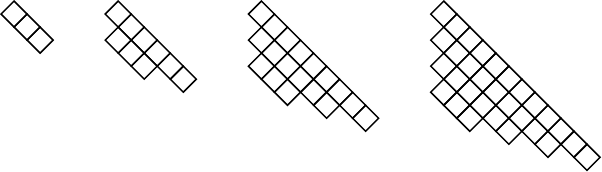}}
  \end{center}
  \caption{The shapes $F_2,F_3,F_4,F_5$, listed from left to right.}\label{Fig14}
\end{figure}

\begin{lemma}\label{lem:C_m}
The region $F_n$ has a unique $3$-ribbon tiling, and this tiling only uses negative bones.  
\end{lemma}

\begin{proof}
The lemma is trivial when $n=1$ since $F_1$ is empty, and it is also obvious when $n=2$. Therefore, we may assume $n\geq 3$ and proceed by induction on $n$. Let $b_1,\ldots,b_{n-1}$ be the leftmost boxes in $F_n$, listed from top to bottom. Suppose we wish to construct a $3$-ribbon tiling of $F_n$. It is clear that we must use a negative bone whose leftmost box is $b_1$. Once we place that bone in the tiling, we readily see that we must use another negative bone whose leftmost box is $b_2$. Continuing in this fashion, we see that each box $b_j$ must be the leftmost box in a negative bone. Once we have placed these $n-1$ negative bones into our tiling, the remaining region has the same shape as $F_{n-1}$, so the proof follows by induction. 
\end{proof}

In what follows, we slightly abuse notation and write $\lambda_n$ for the Young diagram of the partition $\lambda_n$ considered modulo translation (i.e., we allow the lowest point of $\lambda_n$ to be a point other than $0$).

\begin{proposition}\label{prop:3n3n}
There are exactly $(2n)!!$ mountainless tilings of $B_{3n,3n}$. 
\end{proposition} 

\begin{proof}
As above, we can embed $B_{3n,3n}$ inside of $\lambda_N$, where $N=2n-1$. We consider the partition of $\lambda_N$ into the V-shapes $V_1,\ldots,V_N$, which are separated by red borders. The region $B_{3n,3n}$ breaks into three pieces as follows. The first piece is $\lambda_n=V_1\cup\cdots\cup V_n$. The second piece is the part of $B_{3n,3n}\setminus\lambda_n$ lying to the left of the imaginary axis; it has the same shape as the region $F_n$ described above, so we refer to it as $F_n$. The third piece is $F_n'$, which is the part of $B_{3n,3n}\setminus\lambda_n$ lying to the right of the imaginary axis; it is obtained by reflecting $F_n$ across the imaginary axis. Note that $F_n$ and $F_n'$ are not connected to each other. See \cref{Fig16} for an example with $n=3$. 

Suppose $\mathcal T$ is a mountainless tiling of $B_{3n,3n}$. \cref{lem:tiling_the_complement} tells us that there exists a $3$-ribbon tiling of $\lambda_N\setminus B_{3n,3n}$, so we can extend $\mathcal T$ to a (not necessarily mountainless) $3$-ribbon tiling $\mathcal T^*$ of $\lambda_N$. Suppose there is a tile $v$ in $\mathcal T^*$ that crosses the red border between $V_n$ and $V_{n+1}$. It follows from the construction of $\mathcal T^*$ that $v$ must be a tile in $\mathcal T$. Hence, $v$ is not a mountain stone. Applying \cref{lem:valleys_C} to the whole tiling $\mathcal T^*$ of $\lambda_N$, we find that $v$ must be a valley stone that has exactly one box in $\lambda_n$; moreover, this one box has label C. 

We have shown that every tile in $\mathcal T^*$ that crosses the red border between $V_n$ and $V_{n+1}$ must have exactly one box in $\lambda_n$ and that that box must have label C. However, there are equal numbers of boxes labeled A, B, and C lying in $\lambda_n$, and we know that every tile uses exactly one box of each label. It follows that there cannot be any tiles in $\mathcal T^*$ that cross the red border between $V_n$ and $V_{n+1}$. Thus, the tiling $\mathcal T$ can be decomposed into a mountainless tiling of $\lambda_n$, a mountainless tiling of $F_n$, and a mountainless tiling of $F_n'$. \cref{lem:C_m} tells us that $F_n$ has a unique mountainless tiling, and, by symmetry, it follows immediately from the same lemma that $F_n'$ also has a unique mountainless tiling. This shows that the number of mountainless tilings of $B_{3n,3n}$ is the same as the number of mountainless tilings of $\lambda_n$, which we know is $(2n)!!$ by \cref{prop:lambda_n_(2n)!!}. 
\end{proof}

\begin{figure}[ht]
  \begin{center}{\includegraphics[height=5cm]{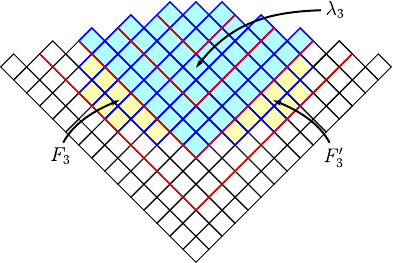}}
  \end{center}
  \caption{The region $B_{9,9}$, which we have embedded into $\lambda_{5}$, breaks into three pieces. The first piece is $\lambda_3$, which is shaded light blue. The other two pieces are $F_3$ and $F_3'$, which are shaded yellow.}\label{Fig16}
\end{figure}

We next want to prove $B_{3n+1,3n+2}$ also has $(2n)!!$ mountainless tilings. In order to do so, we replace \cref{lem:C_m} with the following lemma. 

\begin{lemma}\label{lem:C_m_part_2}
The region $B_{3n+1,3n+2}\setminus\lambda_n$ has a unique mountainless tiling, and this tiling only uses bones.
\end{lemma} 

\begin{proof}
Rather than describe the shape of $B_{3n+1,3n+2}\setminus\lambda_n$ in words, we simply draw it for $n=1,2,3$ in \cref{Fig15}, hoping the pattern is clear. In each of these three cases, we have tiled the region with bones, and we have numbered the bones. Imagine trying to construct a mountainless tiling of $B_{3n+1,3n+2}\setminus\lambda_n$. It is straightforward to check that we must use the bone labeled $1$ in such a tiling. More generally, once we have added the bones numbered $1,\ldots,\ell$ to the tiling, we are forced to add in the bone numbered $\ell+1$. This shows that the given tiling is the only mountainless tiling of $B_{3n+1,3n+2}\setminus\lambda_n$; we trust the reader to see how this argument extends to any choice of a positive integer $n$. 
\end{proof}

\begin{figure}[ht]
  \begin{center}{\includegraphics[height=3.4cm]{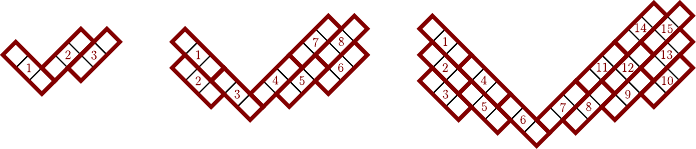}}
  \end{center}
  \caption{A tiling of $B_{3n+1,3n+2}\setminus\lambda_n$ for $n=1$ (left), $n=2$ (middle), and $n=3$ (right). In each tiling, the bones are numbered in such a way that the bone numbered $\ell+1$ is forced to be added to the tiling once the bones numbered $1,\ldots,\ell$ have already been included.}\label{Fig15}
\end{figure}

\begin{proposition}\label{prop:3n+13n+2}
There are exactly $(2n)!!$ mountainless tilings of $B_{3n+1,3n+2}$. 
\end{proposition} 

\begin{proof}
Let us transfer the $(3n+1,3n+2)$-benzel into the square grid and embed it as the region $B_{3n+1,3n+2}$ inside of $\lambda_N$, where $N=2n$. We consider the partition of $\lambda_N$ into the $V$-shapes $V_1,\ldots,V_N$, which are separated by red borders. 

Suppose $\mathcal T$ is a mountainless tiling of $B_{3n+1,3n+2}$. \cref{lem:tiling_the_complement} tells us that there exists a $3$-ribbon tiling of $\lambda_N\setminus B_{3n+1,3n+2}$, so we can extend $\mathcal T$ to a (not necessarily mountainless) $3$-ribbon tiling $\mathcal T^*$ of $\lambda_N$. As in the proof of \cref{prop:3n3n}, one can show that there cannot be any tiles in $\mathcal T^*$ that cross the red border between $V_n$ and $V_{n+1}$. Thus, the tiling $\mathcal T$ can be decomposed into a mountainless tiling of $\lambda_n$ and a mountainless tiling of $B_{3n+1,3n+2}\setminus\lambda_n$. \cref{lem:C_m_part_2} tells us that $B_{3n+1,3n+2}\setminus\lambda_n$ has a unique mountainless tiling. Hence, the number of mountainless tilings of $B_{3n+1,3n+2}$ is the same as the number of mountainless tilings of $\lambda_n$, which we know is $(2n)!!$ by \cref{prop:lambda_n_(2n)!!}. 
\end{proof}

\begin{proof}[Proof of \cref{thm:mountainless}]
When $a+b\not\equiv 0\pmod 3$, the proof follows from the discussion at the beginning of this section. When $a+b\equiv 0\pmod 3$, tilings of the $(a,b)$-benzel using left stones, rising bones, and falling bones correspond to mountainless tilings of $B_{a,b}$, so the proof follows from \cref{prop:b-aatleast2,prop:3n3n,prop:3n+13n+2}. 
\end{proof}

\section{Two Bones and the Right Stone}\label{SecValleyless}

We now wish to enumerate tilings of the $(a,b)$-benzel using two types of bones along with the right stone. Without loss of generality, we may assume the vertical bone is forbidden. Thus, our goal is to prove \cref{thm:valleyless}, which provides the enumeration of tilings using rising bones, falling bones, and right stones whenever $a+b\not\equiv 2\pmod 3$.

\begin{proof}[Proof of \cref{thm:valleyless}]
Suppose $2\leq a\leq b\leq 2a$. If $b=2a$ or $a+b\equiv 1\pmod 3$, then we know by 
\cref{remark:boundaryCase} or \cref{theorem:1mod3} that the $(a,b)$-benzel has a unique stones-and-bones tiling and that this tiling consists entirely of right stones. Hence, we may assume in what follows that $a+b\equiv 0\pmod 3$ and $b<2a$. 

As in \cref{SecMountainless}, let us transfer the $(a,b)$-benzel into the square grid and embed it as the region $B_{a,b}$ inside of $\lambda_N$, where $N=\frac{a+b}{3}-1$. Recall that we think of $\lambda_N$ as the union of $N$ V-shaped regions $V_1,\ldots,V_N$. We say a tiling of a region in the square grid is \dfn{valleyless} if it consists of only mountain stones, positive bones, and negative bones. Tilings of the $(a,b)$-benzel using right stones, rising bones, and falling bones correspond to valleyless tilings of $B_{a,b}$; our goal is to prove that no such tilings exist. 

Suppose by way of contradiction that $\mathcal T$ is a valleyless tiling of $B_{a,b}$. Let $c$ be the number of tiles in $T$ that cross the red border between $V_1$ and $V_2$. According to \cref{lem:valleys_C}, each such tile must be a mountain stone that uses exactly $2$ boxes from $V_1$. Thus, the number of boxes in $V_1$ that belong to tiles that do not cross a red border is $6-2c$; this must be a multiple of $3$, so $c\in\{0,3\}$. Upon inspecting the shape of $V_1$, we immediately see that $c$ cannot be $3$. Hence, there are no tiles in $T$ that cross the red border between $V_1$ and $V_2$. It follows that we can restrict $\mathcal T$ to a valleyless tiling of $V_1$, which is our desired contradiction because there are no valleyless tilings of $V_1$.
\end{proof}

\section*{Acknowledgements}
Colin Defant was supported by the National Science Foundation under Award No.\ 2201907 and by a Benjamin Peirce Fellowship at Harvard University.
James Propp and Ben Young were supported by Simons Foundation Collaboration Grants.
Defant and Rupert Li conducted this research, in part, while visiting the University of Minnesota Duluth REU in 2022 with support from Jane Street Capital; we thank Joe Gallian for providing this wonderful opportunity. We also thank the organizers of the Open Problems in Algebraic Combinatorics 2022 conference, where much of this work was launched. Lastly, we thank Igor Pak for his helpful comments.

\bibliographystyle{amsinit}
\bibliography{ref}

\providecommand{\bysame}{\leavevmode\hbox to3em{\hrulefill}\thinspace}
\providecommand{\MR}{\relax\ifhmode\unskip\space\fi MR }
\providecommand{\MRhref}[2]{%
  \href{http://www.ams.org/mathscinet-getitem?mr=#1}{#2}
}
\providecommand{\href}[2]{#2}
\begin{thebibliography}{10}

\bibitem{ConwayLagarias}
J.~H. Conway and J.~C. Lagarias, \emph{Tiling with polyominoes and
  combinatorial group theory}, J. Combin. Theory Ser. A \textbf{53} (1990),
  no.~2, 183--208. \MR{1041445} \doi{10.1016/0097-3165(90)90057-4}

\bibitem{FominStanton}
S. Fomin and D. Stanton, \emph{Rim hook lattices}, Algebra i Analiz \textbf{9}
  (1997), no.~5, 140--150. \MR{1604377}

\bibitem{JamesKerber}
G. James and A. Kerber, \emph{The representation theory of the symmetric
  group}, Encyclopedia of Mathematics and its Applications, vol.~16, Cambridge
  University Press, 1984. \MR{644144} \doi{10.1017/CBO9781107340732}

\bibitem{LagariasRomano}
J.~C. Lagarias and D.~S. Romano, \emph{A polyomino tiling problem of {T}hurston
  and its configurational entropy}, J. Combin. Theory Ser. A \textbf{63}
  (1993), no.~2, 338--358. \MR{1223689} \doi{10.1016/0097-3165(93)90065-G}

\bibitem{Pak}
I. Pak, \emph{Ribbon tile invariants}, Trans. Amer. Math. Soc. \textbf{352}
  (2000), no.~12, 5525--5561. \MR{1781275} \doi{10.1090/S0002-9947-00-02666-0}

\bibitem{ProppSurvey}
J. Propp, \emph{Enumeration of tilings},
  \href{https://faculty.uml.edu/jpropp/eot.pdf}{https://faculty.uml.edu/jpropp/eot.pdf},
  published in Handbook of Enumerative Combinatorics, M. Bona, editor, CRC
  Press, 2015. \MR{3408702}

\bibitem{Propp2022pentagonal}
J. Propp, \emph{A pentagonal number theorem for tribone tilings},
  \href{https://arxiv.org/abs/2206.04223}{\tt arXiv:2206.04223 [math.CO]}
  (2022).

\bibitem{Propp2022trimer}
\bysame, \emph{Trimer covers in the triangular grid: twenty mostly open
  problems}, \href{https://arxiv.org/abs/2206.06472}{\tt arXiv:2206.06472
  [math.CO]} (2022).

\bibitem{Robinson}
G. {\SortNoop{Robinson}}de B.~Robinson, \emph{Representation theory of the
  symmetric group}, Mathematical Expositions, No. 12, University of Toronto
  Press, Toronto, 1961. \MR{0125885}

\bibitem{StantonWhite}
D.~W. Stanton and D.~E. White, \emph{A {S}chensted algorithm for rim hook
  tableaux}, J. Combin. Theory Ser. A \textbf{40} (1985), no.~2, 211--247.
  \MR{814412} \doi{10.1016/0097-3165(85)90088-3}

\bibitem{Thurston}
W.~P. Thurston, \emph{Conway's tiling groups}, Amer. Math. Monthly \textbf{97}
  (1990), no.~8, 757--773. \MR{1072815} \doi{10.2307/2324578}

\bibitem{VN}
A. Verberkmoes and B. Nienhuis, \emph{Bethe ansatz solution of triangular
  trimers on the triangular lattice}, Phys. Rev. E \textbf{63} (2001), no.~6,
  066122. \doi{10.1103/PhysRevE.63.066122}

\end{thebibliography}

\end{document}